\documentclass[reqno,centertags,11pt,a4paper]{amsart}
\usepackage{amssymb, amsmath, amsthm, bm, latexsym, enumerate, marginnote,  tikz}
\usetikzlibrary{intersections}
\usepackage{color,umoline}
\usepackage{enumitem}
\usepackage[font=small,labelfont=bf]{caption}
\usetikzlibrary{calc}
\usetikzlibrary{decorations.pathmorphing}
\usetikzlibrary{decorations.pathmorphing,calc}
\usepackage{caption}
\usepackage{tikz, caption, subcaption}\usepackage{pgfplots}
\pgfplotsset{compat=1.18}
\usepackage{relsize}
\usepackage[mathscr]{eucal}
\usepackage{verbatim}
\usepackage[draft]{todonotes}
\usepackage{hyperref}

\tikzstyle{EDR}=[draw=lightgray,line width=0pt,preaction={clip, postaction={pattern=north east lines, pattern color=gray}}]
\tikzstyle{EDR1}=[draw=lightgray,line width=0pt,preaction={clip, postaction={pattern=north west lines, pattern color=gray}}]

\definecolor{mygray}{gray}{0.95}

\definecolor{mypink1}{rgb}{1.2,1.1,0.9}

\definecolor{mypink2}{rgb}{1.0,0.95 ,0.9}

\definecolor{mypink3}{rgb}{1.0,0.6,0.7}

\numberwithin{equation}{section}

\newtheorem{theorem}{Theorem}[section]
\newtheorem{definition}[theorem]{Definition}
\newtheorem{lemma}[theorem]{Lemma}
\newtheorem{corollary}[theorem]{Corollary}
\newtheorem{conjecture}[theorem]{Conjecture}
\newtheorem{proposition}[theorem]{Proposition}

\theoremstyle{remark}
\newtheorem{remark}{Remark}

\numberwithin{equation}{section}

\newcommand{\R}{\mathbb R}

\newcommand{\N}{\mathbb N}

\newcommand{\supp}{\operatorname{supp}\,}

\newcommand{\beq}{\begin{equation}}
	\newcommand{\eeq}{\end{equation}}
\newcommand{\beqq}{\begin{equation*}}
	\newcommand{\eeqq}{\end{equation*}}
\newcommand{\ben}{\begin{eqnarray}}
	\newcommand{\een}{\end{eqnarray}}
\newcommand{\beno}{\begin{eqnarray*}}
	\newcommand{\eeno}{\end{eqnarray*}}

\newcommand{\n}{\nabla}

\begin{document}

\title[Restriction Estimates for Eigenfunctions]   
	{Refined $L^p$ restriction estimates for eigenfunctions on Riemannian surfaces}
 \author{Chuanwei Gao}
	\address{School of Mathematical Sciences, Capital Normal University, Beijing 100048, China}
	\email{cwgao@cnu.edu.cn}
	
	\author{Changxing Miao}
	\address{Institute for Applied Physics and Computational Mathematics, Beijing, China}
	\email{miao\textunderscore changxing@iapcm.ac.cn}
	\author{Yakun Xi}
	\address{School of Mathematical Sciences, Zhejiang University, Hangzhou 310027, PR China}
	\email{yakunxi@zju.edu.cn}
	\begin{abstract}
We refine the $L^p$ restriction estimates for Laplace eigenfunctions on a Riemannian surface, originally established by Burq, Gérard, and Tzvetkov \cite{BGT}. First, we establish estimates for the restriction of eigenfunctions to arbitrary Borel subsets of the surface, following the formulation of {Eswarathasan and Pramanik \cite{EP}}. We achieve this by proving a variable-coefficient analogue of a weighted Fourier extension estimate by Du and Zhang \cite{DZ}. Our results naturally unify the $L^p(M)$ estimates of Sogge \cite{sogge1} and the $L^p(\gamma)$ restriction bounds of Burq, Gérard, and Tzvetkov, and are sharp for all $p \geq 2$, up to a $\lambda^\varepsilon$ loss. Second, we derive sharp estimates for the restriction of eigenfunctions to tubular neighborhoods of a curve with nonvanishing geodesic curvature. These estimates are closely related to a variable-coefficient version of the Mizohata--Takeuchi conjecture, providing new insights into eigenfunction concentration phenomena.
\end{abstract}
\maketitle

\section{Introduction}	
\subsection{Restriction of eigenfunctions to fractal measures}

Let $(M,g)$ be a compact, boundaryless, two-dimensional Riemannian manifold with metric $g$.  
Let $e_\lambda$ be an eigenfunction of $-\Delta_g$ corresponding to the eigenvalue $\lambda^2$, satisfying
\begin{equation}
-\Delta_g e_\lambda = \lambda^2 e_\lambda.
\end{equation}
We assume that $e_\lambda$ is normalized so that $\|e_\lambda\|_{L^2(M)} = 1$.

There are several ways to quantify the concentration of eigenfunctions. One approach is to study the size of their $L^p$ norms for $p > 2$. In \cite{sogge1}, Sogge proved that
\begin{equation}\label{soggees}
    \|e_\lambda\|_{L^p(M)} \leq C\lambda^{\delta_2(p)} \|e_\lambda\|_{L^2(M)},
\end{equation}
where
\begin{equation*}
\delta_2(p) =
\begin{cases}
         \tfrac{1}{2}\bigl(\tfrac{1}{2} - \tfrac{1}{p}\bigr), & 2 \leq p \leq 6, \\[6pt]
         \tfrac{1}{2} - \tfrac{2}{p}, & 6 \leq p \leq \infty.
\end{cases}
\end{equation*}
These bounds are sharp on the standard sphere $S^2$ and reveal two distinct types of eigenfunction concentration.

The eigenvalues of $\sqrt{-\Delta_g}$ on $S^2$ are $\sqrt{k^2 + k}$, each with multiplicity $d_k = 2k + 1$. The corresponding eigenspace $\mathcal{E}_k$ consists of spherical harmonics of degree $k$, obtained by restricting homogeneous harmonic polynomials of degree $k$ to the unit sphere. Let $\Pi_k(x, y)$ denote the kernel of the spectral projection operator onto $\mathcal{E}_k$. The $k$-th $L^2$-normalized zonal function centered at $x_0 \in S^2$ is
\[
\mathcal Z_k(y) = \bigl(\Pi_k(x_0, x_0)\bigr)^{-1/2} \Pi_k(x_0, y).
\]
The function $\mathcal Z_k$ is sharply peaked at the poles $\pm x_0$, where it attains the value $\sqrt{d_k / 4\pi}$. It follows that $\|\mathcal Z_k\|_{L^p(S^2)} \sim k^{\delta_2(p)}$ for $p \geq 6$. Hence, on the standard sphere, Sogge’s bounds are optimal in this range of $p$.

A different type of concentration is exhibited by the highest-weight spherical harmonic $\mathcal H_k$, whose mass is concentrated along the equator
\[
\gamma_0 = \{(x_1, x_2, 0) : x_1^2 + x_2^2 = 1\}.
\]
The function $\mathcal H_k$ is the restriction of the harmonic polynomial $k^{1/4}(x_1 + i x_2)^k$ to $S^2 = \{x : |x| = 1\}$. In this normalization, $\mathcal H_k$ has $L^2$ norm comparable to one, and its $L^p$ norms satisfy $\|\mathcal H_k\|_{L^p(S^2)} \sim k^{\delta_2(p)}$ for $2 \leq p \leq 6$.

Another way to study the concentration of eigenfunctions is to examine their restriction to submanifolds. In \cite{BGT}, Burq--Gérard--Tzvetkov investigated the restriction of $e_\lambda$ to submanifolds of compact Riemannian manifolds. In particular, for a smooth curve segment on a Riemannian surface, they proved the following result.

\begin{theorem}[{\cite{BGT}}]\label{cor1}
Let $\gamma \subset M$ be a smooth curve segment. Then there exists a constant $C$ such that
\begin{equation}\label{eqcor}
    \|e_\lambda\|_{L^p(\gamma)} \leq C(1 + \lambda)^{\delta_1(p)} \|e_\lambda\|_{L^2(M)},
\end{equation}
where
\begin{equation}\label{eq:2}
    \delta_1(p) =
    \begin{cases}
        \tfrac{1}{4}, & 2 \leq p \leq 4, \\[6pt]
        \tfrac{1}{2} - \tfrac{1}{p}, & 4 \leq p \leq \infty.
    \end{cases}
\end{equation}
\end{theorem}

Burq--Gérard--Tzvetkov further showed that the exponent in \eqref{eq:2} is sharp on $S^2$ when $\gamma$ is a geodesic. Although this bound cannot be improved on $S^2$, several sharper estimates are known for manifolds with nonpositive curvature (see, for example, \cite{blair2018concerning, xi2017improved}).

In this paper, we study the restriction of eigenfunctions to Borel subsets of $M$, following the framework introduced by Eswarathasan and Pramanik \cite{EP}. To formulate results in this setting, we use the notion of \emph{$\alpha$-dimensional measures}.

\begin{definition}
    A Borel probability measure $\mu$ on $M$ is said to be $\alpha$-dimensional if there exists a constant $C > 0$ such that
\begin{equation}
    \mu(B(x, r)) \leq C r^\alpha, \quad \forall x \in M, 0 < r < \operatorname{inj} M,
\end{equation}
where $B(x, r)$ denotes the geodesic ball of radius $r$ centered at $x$, and $\operatorname{inj} M$ is the injectivity radius of $M$.
\end{definition}

Eswarathasan and Pramanik \cite{EP} obtained restriction estimates for eigenfunctions on arbitrary Borel sets in all dimensions. In two dimensions, for any $\alpha$-dimensional probability measure $\mu$ on $M$, their results yield the following.

\begin{theorem}[\cite{EP}]\label{EPbound}
Let $\alpha \in (0, 2]$, and let $\mu$ be an $\alpha$-dimensional probability measure on $M$. Then for any $\varepsilon > 0$ there exists $C_\varepsilon > 0$ such that
\begin{equation}\label{eq EP}
   \|e_\lambda\|_{L^p(M;\, d\mu)} \leq C_\varepsilon \lambda^{\sigma(\alpha, p) + \varepsilon} \|e_\lambda\|_{L^2(M)},
\end{equation}
where
\[
\sigma(\alpha, p) =
\begin{cases}
    \tfrac{1}{2} - \tfrac{\alpha}{p}, & 0 < \alpha \leq \tfrac{1}{2},\quad 2 \leq p \leq \infty, \\[6pt]
    \tfrac{1}{4}, & \tfrac{1}{2} < \alpha \leq 2,\quad 2 \leq p \leq 4\alpha, \\[6pt]
    \tfrac{1}{2} - \tfrac{\alpha}{p}, & \tfrac{1}{2} < \alpha \leq 2,\quad 4\alpha < p \leq \infty.
\end{cases}
\]
Moreover, \eqref{eq EP} is sharp on $S^2$, up to a $\lambda^\varepsilon$ loss, for $p \geq \max\{2, 4\alpha\}$.
\end{theorem}
We note that the original statement in \cite{EP} is stronger than the streamlined formulation in Theorem \ref{EPbound}, since the factor $\lambda^\varepsilon$ can be removed for all $p\in[2,\infty)$ except at the endpoint $p=4\alpha$, where it can be replaced by a logarithmic loss. As our focus is on the exponents and thresholds in $p$ and $\alpha$, we do not track the precise loss here.

For $p \ge 4\alpha$, these bounds interpolate between Sogge's $L^p(M)$ bounds and the curve restriction estimates of Burq--Gérard--Tzvetkov. In this range,
\[
\sigma(2,p)=\tfrac{1}{2}-\tfrac{2}{p}=\delta_2(p)
\quad\text{and}\quad
\sigma(1,p)=\tfrac{1}{2}-\tfrac{1}{p}=\delta_1(p).
\]
For $\alpha=1$ the critical exponent is $p=4$, agreeing with \eqref{eq:2}. For $\alpha=2$, the critical value $p=8$ does not match Sogge's critical exponent $p=6$. Furthermore, Theorem \ref{EPbound} is not sharp for every pair $(\alpha,p)$ with $\alpha \in (1,2]$ and $p\in[2,4\alpha)$.

These comparisons leave a gap when $\alpha\in(1,2]$ and $2\le p<4\alpha$. Our primary objective is to establish sharp bounds, up to a $\lambda^\varepsilon$ loss, for all $p \ge 2$ and $\alpha \in (0,2]$. Our main result is as follows.

\begin{theorem}\label{theomain}
Let $\alpha \in (0, 2]$, and let $\mu$ be an $\alpha$-dimensional probability measure on $M$. Then for any $\varepsilon > 0$, there exists $C_\varepsilon > 0$ such that
\begin{equation}\label{eq main1}
    \|e_\lambda\|_{L^p(M;\, d\mu)} \leq C_\varepsilon \lambda^{\delta(\alpha, p) + \varepsilon} \|e_\lambda\|_{L^2(M)},
\end{equation}
where 
\begin{equation}\label{eq delta}
    \delta(\alpha, p) = \begin{cases}
    \frac{1}{2} - \frac{\alpha}{p}, & 0 < \alpha \leq \frac{1}{2}, \quad 2 \leq p \leq \infty, \\
    \frac{1}{4}, & \frac{1}{2} < \alpha \leq 1, \quad 2 \leq p \leq 4\alpha, \\
    \frac{1}{2} - \frac{\alpha}{p}, & \frac{1}{2} < \alpha \leq 1, \quad 4\alpha < p \leq \infty, \\
    \frac{1}{4} - \frac{\alpha - 1}{2p}, & 1 < \alpha \leq 2, \quad 2 \leq p \leq 2\alpha + 2, \\
    \frac{1}{2} - \frac{\alpha}{p}, & 1 < \alpha \leq 2, \quad 2\alpha + 2 < p \leq \infty.
\end{cases}
\end{equation}
Moreover, \eqref{eq main1} is sharp on $S^2$, modulo $\lambda^\varepsilon$ losses, for all $p \geq 2$.
\end{theorem}

\begin{figure}[htbp]
\centering

\begin{minipage}{\textwidth}
\centering
\begin{tikzpicture}
 \begin{axis}[
      axis lines=left,
      clip=false,
      axis line style={->, line width=1.2pt, shorten <=-6pt},
      xlabel={$ \frac{1}{p} $},
      ylabel={$ \sigma(\alpha,p) $},
      xlabel style={at={(axis description cs:1,0)}, anchor=west, xshift=0.7em},
      ylabel style={at={(axis description cs:0,1)}, rotate=270, anchor=south, yshift=0.6em},
      domain=0:0.5,
      samples=200,
      legend pos=north east,
      grid=both,
      width=10cm, height=7cm,
      xmin=0, xmax=0.5,
      ymin=0, ymax=0.6,
      tick align=outside,
      tick style={line width=0.6pt}
  ]
    \addplot[blue, thick, domain=0.25:0.5, forget plot] {0.25};
    \addplot[blue, thick, domain=0:0.25] {0.5 - x};
    \addlegendentry{\cite{BGT} $(\alpha=1)$};
    \addplot[only marks, mark=*, mark size=2, blue, forget plot] coordinates {(0.25,0.25)};
    \node[blue] at (axis cs:0.25,0.28) {\scriptsize $p=4$};

    \addplot[green, thick, domain=0.16667:0.5, forget plot] {0.25 - 0.5*x};
    \addplot[green, thick, domain=0:0.16667] {0.5 - 2*x};
    \addlegendentry{\cite{sogge1} $(\alpha=2)$};
    \addplot[only marks, mark=*, mark size=2, green, forget plot] coordinates {(0.16667,0.25-0.5*0.16667)};
    \node[green] at (axis cs:0.16667,0.14) {\scriptsize $p=6$};

    \addplot[orange, dashed, thick, domain=0.125:0.5, forget plot] {0.25};
    \addplot[orange, dashed, thick, domain=0:0.125] {0.5 - 2*x};
    \addlegendentry{\cite{EP} ($\alpha=2$)};
    \addplot[only marks, mark=*, mark size=2, orange, forget plot] coordinates {(0.125,0.25)};
    \node[orange] at (axis cs:0.125,0.28) {\scriptsize $p=8$};

    \addplot[red, dotted, thick, domain=0:0.5, forget plot] {0.5};

    \addplot[
      orange!30,
      fill=orange!30,
      fill opacity=0.3,
      draw=none,
      forget plot
    ]
    plot [domain=0:0.125] (\x,{0.5 - 2*\x})
    -- plot [domain=0.125:0.5] (\x,{0.25})
    -- (0.5, 0.5) -- (0, 0.5)
    -- cycle;

    \addlegendimage{area legend, draw=none, fill=orange!30, fill opacity=0.3}
    \addlegendentry{\cite{EP}};

  \end{axis}
\end{tikzpicture}
\end{minipage}

\vspace{1em}

\begin{minipage}{\textwidth}
\centering
\begin{tikzpicture}
 \begin{axis}[
      axis lines=left,
      clip=false,
      axis line style={->, line width=1.2pt, shorten <=-6pt},
      xlabel={$ \frac{1}{p} $},
      ylabel={$ \delta(\alpha,p) $},
      xlabel style={at={(axis description cs:1,0)}, anchor=west, xshift=0.7em},
      ylabel style={at={(axis description cs:0,1)}, rotate=270, anchor=south, yshift=0.6em},
      domain=0:0.5,
      samples=200,
      legend pos=north east,
      grid=both,
      width=10cm, height=7cm,
      xmin=0, xmax=0.5,
      ymin=0, ymax=0.6,
      tick align=outside,
      tick style={line width=0.6pt}
  ]
    \addplot[blue, thick, domain=0.25:0.5, forget plot] {0.25};
    \addplot[blue, thick, domain=0:0.25] {0.5 - x};
    \addlegendentry{\cite{BGT} $(\alpha=1)$};
    \addplot[only marks, mark=*, mark size=2, blue, forget plot] coordinates {(0.25,0.25)};
    \node[blue] at (axis cs:0.25,0.28) {\scriptsize $p=4$};

    \addplot[green, thick, domain=0.16667:0.5, forget plot] {0.25 - 0.5*x};
    \addplot[green, thick, domain=0:0.16667] {0.5 - 2*x};
    \addlegendentry{\cite{sogge1} $(\alpha=2)$};
    \addplot[only marks, mark=*, mark size=2, green, forget plot] coordinates {(0.16667,0.25-0.5*0.16667)};
    \node[green] at (axis cs:0.16667,0.14) {\scriptsize $p=6$};

    \addplot[red, dotted, thick, domain=0:0.5, forget plot] {0.5};

    \addplot[
      orange!30,
      fill=orange!30,
      fill opacity=0.3,
      draw=none,
      forget plot
    ]
    plot [domain=0:0.16667] (\x,{0.5 - 2*\x})
    -- plot [domain=0.16667:0.5] (\x,{0.25 - 0.5*\x})
    -- (0.5, 0.5) -- (0, 0.5)
    -- cycle;

    \addlegendimage{area legend, draw=none, fill=orange!30, fill opacity=0.3}
    \addlegendentry{Theorem \ref{theomain}};

  \end{axis}
\end{tikzpicture}
\end{minipage}

\caption{Comparison of Theorem \ref{EPbound} (\cite{EP}) and Theorem \ref{theomain}.
Top: graphs of $(1/p,\sigma(\alpha,p))$ from Theorem \ref{EPbound}.
Bottom: graphs of $(1/p,\delta(\alpha,p))$ from Theorem \ref{theomain}.
For each fixed $\alpha$, the theorem yields a piecewise-linear graph, typically two line segments meeting at a critical exponent, and a single segment when no transition occurs.
Aggregating over $\alpha\in(0,2]$ produces distinct envelopes, with the lower panel enlarging the admissible range, especially for $\alpha>1$.}
\label{fig comparison}

\end{figure}

From Theorem \ref{theomain} we obtain sharp exponents, up to the usual $\lambda^\varepsilon$ losses, for all $p \ge 2$ and $\alpha \in (0,2]$. At full dimension $\alpha=2$, Theorem \ref{theomain} identifies the critical Lebesgue exponent $p=2\alpha+2=6$, recovering Sogge’s classical estimate. For intermediate fractal dimensions $\alpha \in (1,2]$ and $2 \le p < 4\alpha$, our bounds strictly improve those of Eswarathasan--Pramanik \cite{EP}. Figure \ref{fig comparison} illustrates this comparison and shows that the sharp range in Theorem \ref{theomain} extends beyond that of Theorem \ref{EPbound} when $\alpha>1$.

As a direct consequence of Theorem \ref{theomain} and Frostman’s lemma, we also obtain $L^p$ bounds for the restriction of eigenfunctions to any Borel subset $E \subseteq M$ with Hausdorff dimension $\dim_{\mathcal{H}}(E)=\alpha \in (0,2]$. Given $\varepsilon>0$, Frostman’s lemma guarantees a $(\alpha-\varepsilon)$-dimensional probability measure $\mu^{(\varepsilon)}$ supported on $E$. Theorem \ref{theomain} then implies the following.

\begin{corollary}\label{coromain}
For any $\varepsilon > 0$, there exists a constant $C_\varepsilon > 0$ such that
\begin{equation}\label{eq mainc}
    \|e_\lambda\|_{L^p(E;\, d\mu^{(\varepsilon)})} \leq C_\varepsilon \lambda^{\delta(\alpha, p) + \varepsilon} \|e_\lambda\|_{L^2(M)},
\end{equation}
where $\delta(\alpha,p)$ is as in \eqref{eq delta}.
\end{corollary}

We remark that the formulation of our fractal $L^p$ restriction estimates differs from Theorem 1.3 of \cite{EP}, where the critical value of $p$ also depends on the choice of $\varepsilon>0$. This discrepancy appears because we do not emphasize the exact size of the $\lambda^\varepsilon$ losses. These losses arise naturally from the multi-scale analysis used in our proof. In contrast, \cite{EP} gives a more precise description of the loss factor. Nonetheless, \eqref{eq mainc} is sharp, modulo the $\lambda^\varepsilon$ losses, for all $p\ge 2$ on the standard sphere for suitable fractal sets $E$. Apart from the $\lambda^\varepsilon$ losses, our results unify the two-dimensional bounds of Sogge \cite{sogge1}, Burq--Gérard--Tzvetkov \cite{BGT}, and Eswarathasan--Pramanik \cite{EP}.

Sogge’s $L^p$ eigenfunction bounds \cite{sogge1} can be viewed as a variable-coefficient analogue of the classical Fourier restriction theorem of Tomas–Stein. Shayya \cite{shayya1} proved a weighted Fourier restriction estimate that sharpens Tomas–Stein. This work partially inspired our proof of Theorem \ref{theomain}, which may be seen as a variable-coefficient counterpart of Shayya’s result.

The proof of Theorem \ref{EPbound} in \cite{EP} uses a $TT^*$ argument. In contrast, our proof of Theorem \ref{theomain} is based on a multi-scale wave-packet analysis. In particular, we use the broad–narrow argument of Bourgain and Guth \cite{BG}. To the best of our knowledge, this is the first application of modern multi-scale wave-packet methods to eigenfunction estimates. This opens a path to applying further techniques, including decoupling and polynomial methods, to spectral problems.

\subsection{Restriction of eigenfunctions to the neighborhood of a curve}
We now study eigenfunction concentration near curves with nonvanishing geodesic curvature.
As a starting point, recall that in \cite{BGT}, Burq--Gérard--Tzvetkov showed that if a curve segment $\gamma$ has nonvanishing geodesic curvature, i.e., if $\gamma$ is parametrized by arc length $s$ and
\begin{equation}
    \langle \nabla_{\gamma'(s)}\gamma'(s), \nabla_{\gamma'(s)}\gamma'(s)\rangle \neq 0,
\end{equation}
then the bound \eqref{eqcor} can be improved. More precisely:

\begin{theorem}[\cite{BGT}]\label{bgt}
Let $\gamma$ be a smooth curve segment in $M$ with nonvanishing geodesic curvature. Then, for $2 \le p \le 4$,
\begin{equation}\label{eq:5}
    \|e_\lambda\|_{L^p(\gamma)} \le C (1+\lambda)^{\frac{1}{3}-\frac{1}{3p}} \|e_\lambda\|_{L^2(M)}.
\end{equation}
Furthermore, \eqref{eq:5} is sharp when $\gamma$ is a small circle on $S^2$.
\end{theorem}

Our second result concerns mass distribution in tubular neighborhoods of such $\gamma$. We show that the average $L^2$ density of $e_\lambda$ in tubes around $\gamma$ decreases in a predictable way as the neighborhood size increases.

\begin{theorem}\label{theo:cur}
Let $\gamma:[0,1]\to M$ be a curve segment with nonvanishing geodesic curvature. If $\tfrac{1}{3} \le s \le \tfrac{1}{2}$, then for any $\varepsilon>0$ there exists $C_\varepsilon>0$ such that
\begin{equation}\label{eq:1/3 to 1/2}
     \|e_\lambda\|_{L^2_{\mathrm{avg}}(T_{\lambda^{-1+s}}(\gamma))} \le C_\varepsilon\, \lambda^{\frac{1-s}{4}+\varepsilon}\, \|e_\lambda\|_{L^2(M)},
\end{equation}
where $T_{\lambda^{-1+s}}(\gamma)$ is the $\lambda^{-1+s}$ tubular neighborhood of $\gamma$, and
\begin{equation*}
     \|f\|_{L^2_{\mathrm{avg}}(T_{\lambda^{-1+s}}(\gamma))} := \Big( \lambda^{1-s} \int_{T_{\lambda^{-1+s}}(\gamma)} |f|^2\, dx \Big)^{1/2}.
\end{equation*}
Moreover, modulo $\lambda^\varepsilon$ losses, \eqref{eq:1/3 to 1/2} is sharp when $\gamma$ is a small circle on $S^2$.
\end{theorem}

The proof of Theorem \ref{theo:cur} uses a straightforward wave-packet analysis together with a geometric observation about how the tubes around $\gamma$ intersect relevant wave-packets.

In comparison, as a direct consequence of \eqref{eq:5}, we have
\begin{equation}\label{eq:0 to 1/3}
     \|e_\lambda\|_{L^2_{\mathrm{avg}}(T_{\lambda^{-1+s}}(\gamma))} \le C\, \lambda^{\frac{1}{6}} \, \|e_\lambda\|_{L^2(M)},
\end{equation}
for every $0 \le s \le \tfrac{1}{2}$. In particular, for $0 \le s \le \tfrac{1}{3}$, \eqref{eq:0 to 1/3} is sharp for the same reason that \eqref{eq:1/3 to 1/2} is sharp when $s=\tfrac{1}{3}$. Our estimate \eqref{eq:1/3 to 1/2} improves this bound for $\tfrac{1}{3} < s \le \tfrac{1}{2}$. In addition, our result identifies the critical scale for concentration near a curve with nonvanishing geodesic curvature as $\lambda^{-2/3}$, which matches the scale at which a Dirichlet eigenfunction can concentrate near the boundary of a domain \cite{SS}.

Theorem \ref{theo:cur} is closely related to the Mizohata--Takeuchi conjecture. We show that a variable-coefficient analogue of this conjecture yields the same bound as \eqref{eq:1/3 to 1/2}, which explains why our result is natural.

Finally, we remark that, as is standard in the study of eigenfunction Lebesgue norms, our bounds apply not only to the eigenfunctions $e_\lambda$ but also to a smoothed spectral projector on $M$.

\subsection{Organization}
We organize the paper as follows. Section \ref{sec 2} reduces the problem to oscillatory integral operators with Carleson--Sj\"olin phases, introduces a wave-packet decomposition, and reduces the proof of Theorem \ref{theomain} to two key statements, Propositions \ref{theo11} and \ref{theo22}. Section \ref{sec 3} proves Proposition \ref{theo11}, establishing weighted $L^2$ estimates via a variable-coefficient bilinear restriction estimate and the Bourgain--Guth method. Section \ref{sec 4} proves Proposition \ref{theo22} and extends the results to $L^q$ for a range of exponents via interpolation and properties of fractal measures. Section \ref{sec 5} proves Theorem \ref{theo:cur} and explores its connection to the Mizohata--Takeuchi conjecture. Finally, Section \ref{sec 6} discusses sharpness, demonstrating optimality through explicit examples and known behavior of spherical harmonics and eigenfunctions.

\subsection{Notation}

	For $r>0$, $B(x, r)$ denotes a ball centered at $x$ with radius $r$. In local coordinates we do not distinguish between geodesic and Euclidean balls at the scales we consider. We abbreviate $B_r$ when the center is not of particular interest.  We use $w_{B_r}$ to denote a Schwartz weight function that is adapted to $B_r$ so that $w_{B_r}\geq 1$ on $B_r$ and decays rapidly away from it.

	For non-negative quantities $A$ and $B$, we shall write $A\lesssim B$, if there is an absolute constant $C>0$ such that $A\le CB$. 
	We shall write $A\approx B$, if $\frac1C B\le A\le CB$ for some absolute constant $C>1$.

	We shall use ${\rm RapDec}(R)$ to denote a term that is rapidly decaying in $R>1$, that is, for any $N\in\N$, there exists a constant $C_N$ such that $|{\rm RapDec} (R)|\le C_N R^{-N}$.

\subsection{Acknowledgements.}This project was supported by the National Key R\&D Program of China: No. 2022YFA1007200,  2022YFA1005700, and 2024YFA1015400. C. Gao was supported by Natural Science Foundation of China grant: No. 12301121.  C. Miao was partially supported by Natural Science Foundation of China grant: No. 12371095. Y. Xi was partially supported by Natural Science Foundation of China under Grant No. 12571107 and 12171424 and the Zhejiang Provincial Natural Science Foundation of China under Grant No. LR25A010001. The first author would like to express gratitude to Jianhui Li and Shaozhen Xu for their insightful discussions. The authors thank the anonymous referees for their careful reading and insightful comments, which have significantly improved the presentation of the paper.

\section{Preliminaries}\label{sec 2}

\subsection{Reduction to oscillatory integral operators with Carleson--Sj\"olin phases}

We begin by applying a standard reduction due to Sogge \cite{soggeFIO}, specifically utilizing the variant presented in \cite{BGT}.

\begin{lemma}
Let $\delta_0 > 0$ be less than half of the injectivity radius of the Riemannian manifold $(M, g)$. Then there exists a function $\chi \in \mathcal{S}(\mathbb{R})$ with $\chi(0) = 1$ such that for any $f \in C^\infty(M)$ and $\lambda > 0$, the operator $\chi_\lambda$ defined by
{\[
\chi_\lambda f(x) := \chi\big(\sqrt{-\Delta_g }-\lambda\big) f(x)
\]}
can be expressed as
\begin{equation}\label{eq refor}
    \chi_\lambda f(x) = \lambda^{\frac{1}{2}} \int_M e^{-i \lambda d_g(x, y)} \alpha(x, y, \lambda) f(y) \, dy + R_\lambda f(x).
\end{equation}
Here, {$d_g(x, y)$ }denotes the Riemannian distance between $x$ and $y$, and the error operator $R_\lambda$ satisfies
\[
\|R_\lambda f\|_{L^\infty(M)} \leq C_N \lambda^{-N} \|f\|_{L^2(M)} \quad \text{for all } N \in \mathbb{N}.
\]
Furthermore, the amplitude $\alpha(x, y, \lambda)$ is a smooth function that vanishes unless $d_g(x, y) \in (\delta_0/2, \delta_0)$ and satisfies
\[
\bigl|\partial_x^{\mathbf a}\partial_y^{\mathbf b}\,\alpha(x,y,\lambda)\bigr|
\le C_{\mathbf a,\mathbf b}
\quad \text{for all } \mathbf a,\mathbf b\in \mathbb{N}^2.
\]
\end{lemma}

By employing a smooth partition of unity, we may assume that the support of $\alpha(x, y, \lambda)$ is contained within $B(x_0, \delta_0/10) \times B(y_0, \delta_0/10)$ for some points $x_0, y_0 \in M$ satisfying $\delta_0/2 < d_g(x_0, y_0) < \delta_0$. We work in geodesic normal coordinates centered at $x_0=(0,0)$, so that the coordinates of $y_0$ are $(0, t_0)$ for some $t_0 \in (\delta_0/2, \delta_0)$.

Next, we express $y$ in polar coordinates around $(0, 0)$ as $y = (t \cos \xi, t \sin \xi)$, where $t = |y|$. It is well-known (see, e.g., Chapter 5 in \cite{soggeFIO}) that for any fixed $t \in (\delta_0/2, \delta_0)$, the phase function
\[
\phi(x, \xi) := -d_g\big(x, (t \cos \xi, t \sin \xi)\big)
\]
satisfies the Carleson--Sj\"olin conditions:
\begin{equation}\label{eq:CS-condition}
    \operatorname{rank} \begin{pmatrix}
        \partial_\xi \partial_{x} \phi(x, \xi) \\
        \partial_\xi^2 \partial_{x} \phi(x, \xi)
    \end{pmatrix} = 2.
\end{equation}

From this point forward, we focus on an oscillatory integral operator associated with a Carleson--Sj\"olin phase function $ \phi(x, \xi) $. Indeed, all of our main results can be rephrased to provide corresponding results for any oscillatory integral operator with a Carleson--Sj\"olin phase.
Given $\lambda > 1$, {we define the rescaled phase function $\phi^\lambda(x, \xi) := \lambda \phi(x/\lambda, \xi),$ and consider the oscillatory integral operator}
\[
\mathcal T^\lambda f(x) := \int e^{i \phi^\lambda(x, \xi)} a^\lambda(x, \xi) f(\xi) \, d\xi.
\]
Here $a^\lambda(x,\xi)$ is a smooth function whose $x$-support is contained in $B(0,\lambda)$.
In our application, it is given by
\[
    a^\lambda(x,\xi) := \alpha\big(x/\lambda,\,(t\cos\xi, t\sin\xi),\,\lambda\big).
\]
Any $L^2 \to L^p$ bound for $\mathcal T^\lambda$ thus implies a corresponding $L^2 \to L^p$ bound for the operator $\chi_\lambda$, after integrating in $t$.

Let $H: \mathbb{R}^2 \to [0, 1]$ be a measurable function. Define $A_\alpha(H)$ by
\begin{equation}\label{fractal}
    A_\alpha(H) := \inf \Big\{ C \ge 0 : \int_{B(x_0,r)} H(x)\, dx \le C r^\alpha \ \text{for all } x_0 \in \mathbb{R}^2 \text{ and } r \ge 1 \Big\}.
\end{equation}
We say that $H$ is an \emph{$\alpha$-dimensional weight} if $A_\alpha(H) < \infty$. The associated weighted $L^p$ norm is defined by
\[
\|f\|_{L^p(B(x_0,r);\, H \, dx)}: = \left( \int_{B(x_0,r)} |f(x)|^p H(x) \, dx \right)^{1/p}.
\]

Given an $\alpha$-dimensional weight $H$, we will show in Section \ref{subsec 2.4 implies 1.3} that the proof of Theorem \ref{theomain} reduces to establishing estimates of the following form:
\[
\|\mathcal T^\lambda f\|_{L^p(B(0, \lambda);\,H  dx)} \leq  C_\varepsilon \lambda^{\beta(\alpha,p)+\varepsilon} \|f\|_{L^2}, \quad 2 \leq p \leq \infty,
\]
for some constant $C_\varepsilon > 0$ independent of $\lambda$ and suitable exponents $\beta(\alpha,p) \geq 0$. 

\begin{proposition}\label{theo11}
{Let $H$ be an $\alpha$-dimensional weight}. For every $\varepsilon > 0$ and $1 \leq R \leq \lambda$, there exists a constant ${ C_\varepsilon}> 0$ such that
\[
\|\mathcal T^\lambda f\|_{L^2(B(0, R);\,H  dx)} \leq {C_\varepsilon} R^{\beta(\alpha)+\varepsilon} \|f\|_{L^2},
\]
where
\[
\beta(\alpha) = \begin{cases}
    0, & 0 < \alpha < \frac{1}{2}, \\
    \frac{\alpha}{2} - \frac{1}{4}, & \frac{1}{2} \leq \alpha < 1, \\
    \frac{\alpha}{4}, & 1 \leq \alpha \leq 2.
\end{cases}
\]
\end{proposition}

\begin{proposition}\label{theo22}
{Let $H$ be an $\alpha$-dimensional weight.} For every $\varepsilon > 0$ and $1 \leq R \leq \lambda$, there exists a constant ${ C_\varepsilon}> 0$ such that
\[
\|\mathcal T^\lambda f\|_{L^q(B(0, R);\,H  dx)} \leq {C_\varepsilon} R^\varepsilon \|f\|_{L^2},
\]
provided that
\[
q \geq \begin{cases}
    2, & 0 < \alpha < \frac{1}{2}, \\
    4\alpha, & \frac{1}{2} \leq \alpha < 1, \\
    2\alpha + 2, & 1 \leq \alpha \leq 2.
\end{cases}
\]
\end{proposition}

By interpolating between Propositions \ref{theo11} and \ref{theo22}, we obtain

\begin{proposition}\label{cormain}
{Let $H$ be an $\alpha$-dimensional weight.} For every $\varepsilon > 0$ and $1 \leq R \leq \lambda$, there exists a constant ${C_\varepsilon}> 0$ such that
\begin{equation}\label{eq mainm}
    \|\mathcal T^\lambda f\|_{L^p(B(0, R);\,H  dx)} \leq {C_\varepsilon} R^{\beta(\alpha,p)+\varepsilon} \|f\|_{L^2},
\end{equation}
where
\[
\beta(\alpha, p) = \begin{cases}
    0, & 0 < \alpha \leq \frac{1}{2}, \quad 2 \leq p \leq \infty, \\
    \frac\alpha p-\frac{1}{4}, & \frac{1}{2} < \alpha \leq 1, \quad 2 \leq p \leq 4\alpha, \\
    0, & \frac{1}{2} < \alpha \leq 1, \quad 4\alpha < p \leq \infty, \\
    \frac{\alpha + 1}{2p}-\frac14, & 1 < \alpha \leq 2, \quad 2 \leq p \leq 2\alpha + 2, \\
    0, & 1 < \alpha \leq 2, \quad 2\alpha + 2 < p \leq \infty.
\end{cases}
\]
\end{proposition}
It is worth noting that the exponent $\delta(\alpha, p)$ in equation \eqref{eq main1} and $\beta(\alpha, p)$ in equation \eqref{eq mainm} are related by the expression
\begin{equation*}
    \delta(\alpha, p) = \beta(\alpha, p) + \frac{1}{2} - \frac{\alpha}{p}.
\end{equation*}
Roughly speaking, the term $ -\frac{\alpha}{p} $ arises from the change of variable $ x \rightarrow \frac{x}{\lambda} $. This will be made precise in Lemma \ref{lefrac}.

\subsection{Wave-packet decomposition}
We perform a wave-packet decomposition for $\mathcal T^\lambda f$. The wave-packet viewpoint for oscillatory and Fourier integral operators goes back to Smith \cite{Smith1998JGA,Smith1998AIF} and to the curvelet-based analysis of Cand\`es--Demanet \cite{CandesDemanet2005}. These decompositions simplify the analysis of wavefronts and compositions, and clarify the operator’s classical action. 

Without loss of generality, assume that the support of $f$ is contained in the open interval $I_0=[0,1]$. In this section we fix a parameter $R \in [1,\lambda]$ and work in the ball $B(0,R)$. First, we cover $I_0$ by $R^\frac{1}{2}$ open subintervals $\theta$ of size $R^{-\frac{1}{2}}$ with bounded overlap. Let $\{\psi_\theta\}$ be a smooth partition of unity subordinate to this cover and write
\begin{equation*}
    f=\sum_\theta f_\theta, \quad f_\theta:=f \psi_\theta.
\end{equation*}
Next, we further decompose $f_\theta$ in the physical side. Cover $\mathbb{R}$ by a collection of finitely overlapping open intervals $\{I_v\}_v$ with center $v$ and length $R^{\frac12}$. Let $\eta_v$ be a smooth partition of unity subordinate to this cover, we further decompose $f$ into
\begin{equation*}
    f=\sum_{\theta,v}f_{\theta,v},\quad f_{\theta,v}:=((f\psi_\theta)^{\wedge}\eta_v)^{\vee}.
    \end{equation*}
    By Plancherel's theorem and the fact that $\{\psi_\theta\}$ and $\{\eta_v\}$ form partitions of unity in frequency and in space, respectively, we obtain the following almost-orthogonality estimate. \begin{lemma}[Almost orthogonality]\label{ortho}
\begin{equation*}
  \Big\|\sum_{\theta,v} f_{\theta,v}\Big\|_{L^2}^2 \approx \sum_{\theta,v}\|f_{\theta,v}\|_{L^2}^2,
\end{equation*}
where the implicit constants depend only on the uniform overlap bounds of the frequency and spatial partitions $\{\psi_\theta\}$ and $\{\eta_v\}$, and are therefore independent of $R$.
\end{lemma}
Let $\xi_\theta$ denote the center point of the interval $\theta$.
{Since $\phi$ is a Carleson--Sj\"olin phase, the associated nondegeneracy implies that the map $x_1  \mapsto \partial_\xi \phi(x_1,x_2,\xi_\theta)$ is locally invertible. By the implicit function theorem, for each $v \in [0,\lambda]$ there exists a smooth function $\gamma_\theta(v/\lambda,x_2)$ such that
\begin{equation}\label{gamma_theta}
\partial_{\xi}\phi(\gamma_\theta(v/\lambda,x_2),x_2,\xi_\theta) = v/\lambda .
\end{equation}
We denote the corresponding curve by
\[
\Gamma_{\xi_\theta,v}:=\{(\gamma_\theta(v/\lambda,x_2),x_2):\, |x_2|\leq 1\}.
\]
We define $\gamma_\theta^\lambda(v,x_2):=\lambda \gamma_\theta(v/\lambda,x_2/\lambda)$ and the rescaled curve $\Gamma_{\xi_\theta,v}^\lambda$ as
\begin{equation}\label{Gamma}\Gamma_{\xi_\theta,v}^\lambda:=\{(\gamma_\theta^\lambda(v,x_2),x_2):\, |x_2|\leq R\}.
\end{equation}
Given $0<\varepsilon\ll 1$, we consider a curved tube $T_{\theta,v}$ given by
\begin{equation*}
  T_{\theta,v}:= \{(x_1,x_2): |x_1-\gamma_{\theta}^\lambda(v,x_2)|\leq R^{\frac{1}{2}+\varepsilon}, |x_2|\leq R\}.
\end{equation*}}
The next lemma shows that for each $\theta,v$, $\mathcal T^\lambda f_{\theta,v}$ is essentially supported on the tube $T_{\theta,v}$.
\begin{lemma}\label{lem tube}
If $x\in B(0,R)\setminus T_{\theta,v}$, then
\[
\big|\mathcal T^\lambda f_{\theta,v}(x)\big|={\rm RapDec}(R)\,\|f\|_{L^2}.
\]
\end{lemma}

\begin{proof}
Recall that
\[
\mathcal T^\lambda f_{\theta,v}(x)=\int e^{i \phi^\lambda(x,\xi)}a^\lambda(x,\xi)f_{\theta,v}(\xi)\,d\xi,
\qquad
f_{\theta,v}=\big((f\psi_\theta)^{\wedge}\eta_v\big)^{\vee}.
\]
By the definition of $f_{\theta,v}$,
\[
   \mathcal T^\lambda f_{\theta,v}(x)={\rm RapDec}(R), \quad \text{if }\xi\notin B\big(\xi_\theta,R^{-\frac12+\frac\varepsilon2}\big).
\]
By Fourier inversion, we have
\[
\mathcal T^\lambda f_{\theta,v}(x)=\iint e^{i(\phi^\lambda(x,\xi)-y\xi)}\,a^\lambda(x,\xi)\,(f\psi_\theta)^{\wedge}(y)\,\eta_v(y)\,d\xi\,dy.
\]
By our construction,
\[
   \mathcal T^\lambda f_{\theta,v}(x)={\rm RapDec}(R), \quad \text{if }y \notin B\big(v,R^{\frac12+\frac\varepsilon2}\big).
\]
Thus, we may assume $y\in B(v,R^{1/2+\varepsilon}),\,\xi \in B(\xi_\theta,R^{-1/2+\varepsilon})$ by inserting smooth bump functions confining $y$ and $\xi$ to these scales.

Noting that 
\[
    \partial_\xi \phi^\lambda\big(\gamma_{\theta}^\lambda(v,x_2),x_2,\xi_\theta\big)=v,
\]
if $x\notin T_{\theta,v}$, then by the mean value theorem in $x_1$ and the Carleson--Sj\"olin nondegeneracy,
\[
    \big|\partial_\xi \phi^\lambda(x,\xi)-y\big|\gtrsim R^{1/2+\varepsilon}
\]
uniformly for $(y,\xi)$ in the above supports.

We integrate by parts using
\[
L = \frac{1}{i(\partial_\xi \phi^\lambda(x,\xi)-y)}\,\partial_\xi,
\qquad L\, e^{i(\phi^\lambda(x,\xi)-y\xi)}=e^{i(\phi^\lambda(x,\xi)-y\xi)}.
\]
Each application of $L$ contributes a factor $|\partial_\xi \phi^\lambda(x,\xi)-y|^{-1}\lesssim R^{-1/2-\varepsilon}$, but also differentiates the amplitude/cutoff, whose $\xi$-derivative costs at most $R^{1/2-\varepsilon}$. Hence each integration-by-parts step yields a net factor $R^{-2\varepsilon}$. After $N$ iterations,
\[
\big|\mathcal T^\lambda f_{\theta,v}(x)\big|\le C_N R^{-2N\varepsilon}\,\|f\|_{L^2}.
\]
Since $N$ is arbitrary, this is ${\rm RapDec}(R)\,\|f\|_{L^2}$.
\end{proof}
\subsection{Locally constant property}
Roughly speaking, the locally constant property says that if the Fourier transform of a function $f$ is supported in a ball of radius $r$, then $f$ can be treated as a constant in any ball of radius $1/r$. The fact is extensively used in the research of restriction theory; one may refer to \cite{BG,Guth2,DZ} for details. Let us formulate the locally constant property in our settings. 

\begin{lemma}\label{leloc}
The Fourier support of $\mathcal T^\lambda f$ is localized, in the sense that there exists a constant $C_0\geq1$ such that
\begin{equation}\label{local12}
    |(\mathcal T^\lambda f)^{\wedge}(\eta)|\leq {\rm RapDec}(\lambda)\|f\|_{L^2},\quad \eta \notin B(0,C_0).
\end{equation}
\end{lemma}
\begin{proof}
It is a straightforward consequence of an integration by parts argument. Indeed, 
\begin{equation*}
(\mathcal T^\lambda f)^{\wedge}(\eta)=\iint e^{-i\eta x+i\phi^\lambda(x,\xi)}a^\lambda(x,\xi)f(\xi)\,d\xi dx.
\end{equation*}
Since $x\in B(0,\lambda),\, \xi\in B(0,1)$, it follows that there exists a constant $C_0>0$ such that 
\begin{equation*}
    \{\partial_x\phi^\lambda(x,\xi): x\in B(0,\lambda),\xi \in B(0,1)\}\subset B(0,C_0/2).
\end{equation*}
Integrating by parts in $x$, we have \eqref{local12}.
\end{proof}
As a direct consequence of Lemma \ref{leloc}, we have
\begin{lemma}[Locally constant property]\label{localc}
    For $1\leq p<\infty$, there exists a constant $C_p$ and a Schwartz weight function $w_{B(0,1)}$ such that
\begin{equation*}
   \|\mathcal T^\lambda f\|_{L^\infty(B(0,1))}\leq C_p  \|\mathcal T^\lambda f\|_{L^p(w_{B(0,1)})}+ {\rm RapDec}(\lambda)\|f\|_{L^2(B(0,1))},
\end{equation*}
where 
\[
\|\mathcal{T}^\lambda f\|_{L^p(w_{B(0,1)})}
:= \Big( \int_{\mathbb{R}^2} |\mathcal{T}^\lambda f(x)|^p\, w_{B(0,1)}(x)\, dx \Big)^{1/p}.
\]
\end{lemma}
\begin{proof}
Let $\psi$ be a radial Schwartz function such that $\widehat\psi=1$ on $B(0,C_0)$.
By Lemma \ref{leloc}, we have 
\begin{equation*}
    (\mathcal T^\lambda f)^{\wedge}(\eta)=(\mathcal T^\lambda f)^{\wedge}(\eta)\widehat{\psi}(\eta)+{\rm RapDec}(\lambda).
\end{equation*}
Therefore, if $x\in B(0,1)$, we have
\begin{multline*}
    |\mathcal T^\lambda f(x)|\leq |\mathcal T^\lambda f\ast \psi(x)|+{\rm RapDec}(\lambda)\|f\|_{L^2(B(0,1))}\\\leq\int |\mathcal T^\lambda f|(y)|\psi (x-y)|dy+{\rm RapDec}(\lambda)\|f\|_{L^2(B(0,1))}.
\end{multline*}
Now we choose a nonnegative Schwartz weight $w_{B(0,1)}$ such that
\[
\sup_{x\in B(0,1)}|\psi(x-y)|\le w_{B(0,1)}(y)\quad\text{for all }y.
\]
Then, for every $x\in B(0,1)$,
\[
|\mathcal T^\lambda f(x)|
\le \int |\mathcal T^\lambda f(y)|\, w_{B(0,1)}(y)\,dy
+ {\rm RapDec}(\lambda)\,\|f\|_{L^2(B(0,1))}.
\]
Applying H\"older’s inequality with exponents $p$ and $p'$, we obtain
\[
\int |\mathcal T^\lambda f(y)|\, w_{B(0,1)}(y)\,dy
\le \|\mathcal T^\lambda f\|_{L^p(w_{B(0,1)})}\,\|w_{B(0,1)}\|_{L^{p'}(\mathbb R^2)}.
\]
Note that $\|w_{B(0,1)}\|_{L^{p'}}$ is finite and depends only on $p$. This completes the argument.
\end{proof}

\subsection{Proposition \ref{cormain} implies Theorem \ref{theomain}}\label{subsec 2.4 implies 1.3}
Recall that, we say a compactly supported probability measure $\mu$ is $\alpha$-dimensional if there exists a constant $C>0$, such that
    \begin{equation}
       \mu(B(x,r))\leq C r^\alpha, \quad \forall x\in \mathbb{R}^2,  r>0.
    \end{equation}
The following lemma allows us to bound $\|\mathcal T^\lambda f\|_{L^p(M;\, d\mu)}$ by $\|\mathcal T^\lambda f\|_{L^p(B(0,\lambda);H dx)}$, where $H$ is a suitable $\alpha$-dimensional weight function. 
\begin{lemma}\label{lefrac}
    Let $\mu$ be an $\alpha$-dimensional measure. 
For every function $f\in L^2(B(0,1))$ and all $p\geq 1$, there exists an $\alpha$-dimensional weight function $H(x)$ such that
    \begin{equation*}
        \int |\mathcal T^\lambda f(\lambda x)|^p d\mu(x)\leq C_p \lambda^{-\alpha}\int |\mathcal T^\lambda f|^p H(x)dx+{\rm RapDec}(\lambda)\|f\|^p_{L^2(B(0,1))}    \end{equation*}
         where $C_p$ is a constant that depends only on $p$.
\end{lemma}
\begin{proof}
Let $\psi$ denote the Schwartz function from the proof of Lemma \ref{localc}, which was used to localize $(\mathcal T^\lambda f)^{\wedge}$ to unit spatial scale.
As in the proof of Lemma \ref{localc}, we have
\[
|\mathcal T^\lambda f(x)| \le |\mathcal T^\lambda f * \psi|(x)
   + {\rm RapDec}(\lambda)\,\|f\|_{L^2(B(0,1))}.
\]
Hence, we may replace $\mathcal T^\lambda f$ by $\mathcal T^\lambda f * \psi$.
By H\"older's inequality, we have
\begin{equation}
\begin{aligned}
    \int \big|\mathcal T^\lambda f(\lambda x)\big|^p \, d\mu(x)
    &\le \int \Big|\int \mathcal T^\lambda f(y)\,\psi(\lambda x-y)\, dy \Big|^p \, d\mu(x)
      + {\rm RapDec}(\lambda)\,\|f\|_{L^2(B(0,1))}^p \\
    &\le \int \int \big|\mathcal T^\lambda f(y)\big|^p \, \big|\psi(\lambda x-y)\big| \, dy \,
      \Big(\int \big|\psi(\lambda x-z)\big| \, dz \Big)^{\frac{p}{p'}} d\mu(x) \\
    &\hspace{2em}+ {\rm RapDec}(\lambda)\,\|f\|_{L^2(B(0,1))}^p \\
    &\le C_p \int \int \big|\mathcal T^\lambda f(y)\big|^p \, \big|\psi(\lambda x-y)\big| \, dy \, d\mu(x)
      + {\rm RapDec}(\lambda)\,\|f\|_{L^2(B(0,1))}^p .
\end{aligned}
\end{equation}
Here we used that $\int \big|\psi(\lambda x-z)\big|\,dz=\int |\psi(u)|\,du$ is a constant independent of $x$ and $\lambda$.
We define the $\alpha$-dimensional weight function $H(y)$ as
\begin{equation*}
H(y):=\lambda^\alpha\int |\psi(\lambda x-y)|d\mu(x)=\lambda^\alpha[\mu^\lambda*|\psi|](y),
\end{equation*}
where $\mu^\lambda$ is the pushforward measure of $\mu$ under the dilation map $x\to\lambda x$ defined as
\[
\int h(x) \, d\mu^\lambda(x) := \int h(\lambda x) \, d\mu(x).
\]
Given the rapid decay of $|\psi|$, it is then routine to check that for every $r\ge1$,
\begin{equation*}
   \int_{B(x,r)} H(y) dy\lesssim\lambda^\alpha\mu^\lambda(B(x,r))\lesssim \lambda^\alpha \Big(\frac{r}{\lambda}\Big)^\alpha=r^\alpha. 
\end{equation*}
\end{proof}
\begin{proof}[Proof of Theorem \ref{theomain}]We are now ready to prove Theorem \ref{theomain} assuming Proposition \ref{cormain}. By the definition of $\chi_\lambda$, we have
\begin{equation}
    e_\lambda = \chi_\lambda e_\lambda + R_\lambda e_\lambda.
\end{equation}
The contribution of the $R_\lambda$ term is negligible. Finally, by combining Proposition \ref{cormain} and Lemma \ref{lefrac}, we obtain the desired estimates for $\|\chi_\lambda e_\lambda\|_{L^p(M;\, d\mu)}$, thus completing the proof of Theorem \ref{theomain}.
\end{proof}
\subsection{Variable-coefficient bilinear restriction estimate} Finally, we record a bilinear estimate that will be used to control the contribution of broad points in the next section.
Let $\tau_1,\tau_2$ be two disjoint intervals contained in $[0,1]$. Assume $\supp f_{\tau_i} \subset \tau_i, i=1,2$, we say the transversality condition holds if there exists a constant $\nu>0$ such that 
\begin{equation*}
    |{\rm det}(\partial_\xi\partial_x\phi(x,\xi_1),\partial_\xi\partial_x\phi(x,\xi_2))|\geq \nu>0,\quad \text{for all }\xi_i \in \tau_i,i=1,2.
\end{equation*}
\begin{remark}
One easily verifies that if $|\xi_1 - \xi_2| \gtrsim K^{-1}$, 
then the transversality condition holds with $\nu \sim K^{-1}$, 
as is evident from the toy model $\phi(x,\xi)=x_1\xi+x_2\xi^2$.

\end{remark}

\begin{theorem}[Theorem 6.2 in \cite{BCT}]\label{multit}
   Let $q\geq 4$. Suppose that the  transversality condition holds, then for each $\varepsilon_1>0,$ there exists constants $ C_{\varepsilon_1},\, C>0$
   \begin{equation}\label{multi}
       \|\mathcal T^\lambda f_{\tau_1} \mathcal T^\lambda f_{\tau_2}\|_{L^{\frac{q}{2}}(B(0,R))}\leq C_{\varepsilon_1}\nu^{-C} R^{\varepsilon_1} \|f_{\tau_1}\|_{L^2}\|f_{\tau_2}\|_{L^2}.
   \end{equation}
\end{theorem}

\section{Proof of  Proposition \ref{theo11}}\label{sec 3}
In this section, we prove Proposition \ref{theo11}. The key input is a weighted estimate for the oscillatory integral operator $\mathcal T^\lambda$. We follow the Bourgain--Guth argument: we partition $B(0,R)$ according to the local configuration of wave-packets at a point. A point is \emph{broad} if at least two intersecting wave-packets with significant contributions have well-separated directions; otherwise, if the directions of all wave-packets with significant contributions concentrate near a single point of $S^1$, it is \emph{narrow}. We control the broad contribution via a bilinear estimate and treat the narrow contribution using an induction-on-scales argument.

\subsection{Summary of notation used in the multiscale analysis}
In this subsection we record the definitions and notation for the families of geometric objects—at various scales—used in the proof. This summary is intended to guide the reader through the argument.

First, fix the parameters $R,\lambda,\varepsilon$ with $1 \le R \le \lambda$ and $0<\varepsilon\ll 1$. 
Set $K = R^{\kappa}$ with $\kappa = \varepsilon^{2}$, $R_{1} = R/K^{2}$, and $K_{1} = R_{1}^{\kappa}$. Let $\rho$ range over $[1, R^{1/2}]$. See Table \ref{tab:multiscale-params}.

\begin{table}[htbp]
\centering
\caption{Parameters and scales used in the proof.}
\label{tab:multiscale-params}
\begin{tabular}{ccc}
\hline
\textbf{Symbol} & \textbf{Description} & \textbf{Range/Notes} \\
\hline
$\lambda$ & Frequency scale & $\lambda \ge 1$ \\
$R$ & Spatial radius & $1 \le R \le \lambda$ \\
$\varepsilon$ & Small parameter & $0 < \varepsilon \ll 1$ \\
$\rho$ & Intermediate scale & $1 \le \rho \le R^{1/2}$ \\
$\kappa$ & $\varepsilon^2$ & Fixed once $\varepsilon$ is fixed \\
$K$ & $R^\kappa$ & Auxiliary scale \\
$R_1$ & $R/K^2$ & Rescaled spatial radius \\
$K_1$ & $R_1^\kappa$ & Rescaled auxiliary scale \\
$\alpha$ & Dimension of the weight & $0 < \alpha \le 2$; see \eqref{fractal} \\
$\beta(\alpha,p)$ & Exponent in weighted estimate & See \eqref{eq mainm} \\
$\delta(\alpha,p)$ & Exponent in main theorem & See \eqref{eq delta} \\
\hline
\end{tabular}
\end{table}

We work in a ball $B(x_0,R)\subset B(0,\lambda)$. Without loss of generality, we may assume that $x_0=0$ and write $B_R=B(0,R)$. We use $\mathcal{Q}=\{Q\}$ to denote a finitely overlapping collection of squares of side length $K^2$. 

By \eqref{gamma_theta}, for each pair $(a,b)\in[0,1]\times[0,R]$ there exists a function $\gamma_{a,b}=\gamma_{a,b}(x_2)$ such that
\begin{equation}
    \partial_\xi \phi\big(\gamma_{a,b}(x_2),x_2,a\big)=\frac{b}{\lambda}.
\end{equation}
We denote its rescaling by $\gamma^\lambda_{a,b}({}\cdot{}):=\lambda\,\gamma_{a,b}({}\cdot/\lambda)$ and set
\[
\Gamma_{a,b}^\lambda:=\big\{ \big(\gamma_{a,b}^\lambda(x_2),\,x_2\big) \colon |x_2|\le R \big\}.
\]
This is consistent with our definition of $ \Gamma_{\xi_\theta,v}^\lambda $ as given in \eqref{Gamma}.
We shall refer to $a$ as the direction of the curve.

For every $1\le \rho \le R^{1/2}$ and $(a,b)\in[0,1]\times[0,R]$, let $I_{a,b}$ denote the $K^2\rho^{1-4\kappa}$-neighborhood of $\Gamma_{a,b}^\lambda$. 
We decompose $I_{a,b}$ into a finitely overlapping collection $\mathcal{I}_{a,b}$ of curved sub-tubes $I$, each of length $K^2\rho^{2-4\kappa}$ and the same width $K^2\rho^{1-4\kappa}$; see Figure \ref{curI} for an illustration.
\begin{figure}[htbp]
\centering
\begin{tikzpicture}[scale=1.0]

\def\w{0.50}   
\def\yA{2.1}   
\def\yB{4.1}   
\def\H{6.0}    
\def\C{      (0.00,0)
             (-0.16,1)
             (-0.22,2)
             (-0.14,3)
             (0.02,4)
             (0.18,5)
             (0.30,6) }

\coordinate (Lbot) at (-\w+0.00,0);
\coordinate (Ltop) at (-\w+0.30,6);
\coordinate (Rtop) at ( \w+0.30,6);
\coordinate (Rbot) at ( \w+0.00,0);

\fill[gray!18]
  plot [smooth, tension=1.15]
    coordinates {(-\w+0.00,0) (-\w-0.16,1) (-\w-0.22,2) (-\w-0.14,3)
                 (-\w+0.02,4) (-\w+0.18,5) (-\w+0.30,6)}
  to[out=0,in=180] (Rtop)
  --
  plot [smooth, tension=1.15]
    coordinates {( \w+0.30,6) ( \w+0.18,5) ( \w+0.02,4) ( \w-0.14,3)
                 ( \w-0.22,2) ( \w-0.16,1) ( \w+0.00,0)}
  to[out=180,in=0] (Lbot) -- cycle;

\draw[black, line width=0.8pt, line cap=round]
  plot [smooth, tension=1.15]
    coordinates {(-\w+0.00,0) (-\w-0.16,1) (-\w-0.22,2) (-\w-0.14,3)
                 (-\w+0.02,4) (-\w+0.18,5) (-\w+0.30,6)};
\draw[black, line width=0.8pt, line cap=round]
  plot [smooth, tension=1.15]
    coordinates {( \w+0.00,0) ( \w-0.16,1) ( \w-0.22,2) ( \w-0.14,3)
                 ( \w+0.02,4) ( \w+0.18,5) ( \w+0.30,6)};
\draw[black, line width=0.8pt, line cap=round] (Ltop) to[out=0,in=180] (Rtop);
\draw[black, line width=0.8pt, line cap=round] (Rbot) to[out=180,in=0] (Lbot);

\draw[dashed, gray, thick]
  plot [smooth, tension=1.15] coordinates {\C};

\begin{scope}
  \clip
    plot [smooth, tension=1.15]
      coordinates {(-\w+0.00,0) (-\w-0.16,1) (-\w-0.22,2) (-\w-0.14,3)
                   (-\w+0.02,4) (-\w+0.18,5) (-\w+0.30,6)}
    to[out=0,in=180] (Rtop) --
    plot [smooth, tension=1.15]
      coordinates {( \w+0.30,6) ( \w+0.18,5) ( \w+0.02,4) ( \w-0.14,3)
                   ( \w-0.22,2) ( \w-0.16,1) ( \w+0.00,0)}
    to[out=180,in=0] (Lbot) -- cycle;
  \fill[red!55] (-1,\yA) rectangle (1,\yB);
\end{scope}

\node[red!70!black, font=\large] at (1.05,3) {$I$};

\draw[<->, thick] (1.55,0) -- (1.55,\H);
\node[right] at (1.55,3) {$R$};
\node[right] at (0.5,\H+0.2) {$I_{a,b}$};

\draw[<->, thick, red!70!black] (-1.35,\yA) -- (-1.35,\yB);
\node[left, red!70!black] at (-1.35,3) {$K^2\rho^{2-4\kappa}$};

\draw[<->, thick] (-\w,0) -- (\w,0);
\node[below] at (0,0) {$K^2\rho^{1-4\kappa}$};

\end{tikzpicture}

\caption{Curved tubes $I_{a,b}$ and a sub-tube $I$.
The tube $I_{a,b}$ has length $R$ and width $K^2\rho^{1-4\kappa}$,
while $I$ denotes a sub-tube of $I_{a,b}$ with length $K^2\rho^{2-4\kappa}$ and the same width.
The dashed curve is $\Gamma_{a,b}^\lambda$.}
\label{curI}
\end{figure}

In this way we obtain the collection
\[
\mathcal{I}
:= \bigl\{\, I \colon I \in \mathcal{I}_{a,b}\text{ is a length } K^2\rho^{2-4\kappa} \text{ sub-tube of } I_{a,b},
\ \text{for some } (a,b)\in[0,1]\times[0,R] \,\bigr\}.
\]
Similarly, we define the collection of $R/\rho$-neighborhoods of $\Gamma_{a,b}^\lambda$ by
\[
\mathcal{O}
:= \bigl\{\, O \colon
O \text{ is the $R/\rho$-neighborhood of } \Gamma_{a,b}^\lambda,\ \text{for some }(a,b)\in[0,1]\times[0,R] \,\bigr\}.
\]
For simplicity, we shall continue to write $O$ for an element of $\mathcal{O}$.

\begin{remark}
For the translation-invariant model phase $\phi(x,\xi)=x_1\xi+x_2\xi^2$, the curve $\Gamma_{a,b}^\lambda$ becomes a line segment of length $R$, lying on the line
\[
\{(x_1,x_2)\colon x_1=b-2a\,x_2\}.
\]
In this case, the elements of $\mathcal{I}$ are  rectangles of dimensions $K^2\rho^{1-4\kappa}\times K^2\rho^{2-4\kappa}$, while the elements of $\mathcal{O}$ are rectangles of dimensions $(R/\rho)\times R$.
\end{remark}

For later use, we also introduce the notation $Q'$, $I'$, $O'$, and $T'_{\tau,v}$ for curved tubes at several scales produced by parabolic rescaling. Recall that $R_1=R/K^2$ and $K_1=R_1^{\kappa}$.

We use $Q'_{a,b}$ to denote the $K\,K_1^2$–neighborhood of $\Gamma_{a,b}^\lambda$.
Each tube $Q'_{a,b}$ is further decomposed along its longitudinal direction into a finitely overlapping collection $\{Q'\}$ of shorter tubes of length $K^2K_1^2$. We denote the collection of all such $Q'$ over all possible  $(a,b)\in[0,1]\times[0,R]$ by $\mathcal{Q}'$.

Similarly, let $I'_{a,b}$ denote the $K\,K_1^2\rho^{\,1-4\kappa}$–neighborhood of $\Gamma_{a,b}^\lambda$.
We partition $I'_{a,b}$ into a finitely overlapping family of shorter tubes of length $K^2K_1^2\rho^{\,2-4\kappa}$, and denote the collection of all such $I'$ over all possible $(a,b)\in[0,1]\times[0,R]$ by $\mathcal{I}'$.
For convenience, we write $I'$ for an element of $\mathcal{I}'$.

Next, let $\mathcal{O}'$ denote the collection of $R/(K\rho)$-neighborhoods of the curve segments $\Gamma_{a,b}^\lambda$. 
For convenience, we write $O'$ for an element of $\mathcal{O}'$.

Finally, let $\tau\subset[0,1]$ be an interval of length $K^{-1}$ centered at $\xi_\tau$. 
We choose a collection of parameters $\{v\}\subset[0,R]$ that are $R/K$–separated, and denote by $T_{\tau,v}'$ the $R/K$–neighborhood of the curve $\Gamma_{\xi_\tau,v}^\lambda$. 
We write $\mathfrak{T}_\tau'$ for the collection of all such tubes.

\begin{table}[h]
\centering
\caption{Summary of the main geometric objects used in the multiscale analysis.}
\renewcommand{\arraystretch}{1.5}
\begin{tabular}{c c c c c}
\hline
\textbf{Symbol} & \textbf{Description} & \textbf{Collection} & \textbf{Width} & \textbf{Length} \\
\hline
$Q$ & Square & $\mathcal{Q}$ & $K^2$ & $K^2$ \\
$I$ & Sub-tube of $I_{a,b}$ & $\mathcal{I}$ & $K^2\rho^{1-4\kappa}$ & $K^2\rho^{2-4\kappa}$ \\
$O$ & $R/\rho$–neighborhood of $\Gamma_{a,b}^\lambda$ & $\mathcal{O}$ & $R/\rho$ & $R$ \\
$Q'$ & Rescaled tube & $\mathcal{Q}'$ & $K K_1^2$ & $K^2 K_1^2$ \\
$I'$ & Sub-tube of $I'_{a,b}$ & $\mathcal{I}'$ & $K K_1^2\rho^{1-4\kappa}$ & $K^2 K_1^2\rho^{2-4\kappa}$ \\
$O'$ & $R/(K\rho)$–neighborhood of $\Gamma_{a,b}^\lambda$ & $\mathcal{O}'$ & $R/(K\rho)$ & $R$ \\
$T_{\tau,v}'$ & Tube with base $v$ and direction $\xi_\tau$ & $\mathfrak{T}_\tau'$ & $R/K$ & $R$ \\
\hline
\end{tabular}
\end{table}

\subsection{The case $1\le \alpha \le 2$}
We now prove Proposition \ref{theo11} in the range $1\le \alpha \le 2$ via the Bourgain--Guth argument.

\begin{proposition}\label{prop 1,2}
Let $1\le \alpha\le 2$. Let $H$ be an $\alpha$-dimensional weight. For every $\varepsilon>0$ and $1\le R\le \lambda$, there exists a constant $C_\varepsilon>0$ such that
\begin{equation}\label{maines1}
    \|\mathcal T^\lambda f\|_{L^2(B(0,R); H\,dx)}
    \le C_\varepsilon A_\alpha(H)^\frac12 R^{\frac{\alpha}{4}+\varepsilon} \|f\|_{L^2}.
\end{equation}
\end{proposition}

Our proof of \eqref{maines1} follows the approach of Zorin-Kranich \cite{ZK}. The estimate \eqref{maines1} can be viewed as a variable-coefficient analogue of Theorem 1.6 in \cite{DZ}. In particular, when $\alpha=1$, Du--Zhang’s Theorem 1.6 yields sharp results on Carleson’s pointwise convergence for the Schr\"odinger operator in $\mathbb{R}^{1+1}$. Theorem 1.6 in \cite{DZ} also covers higher dimensions. We plan to study the corresponding variable-coefficient extensions in a sequel.

\subsubsection{Reductions}
First, decompose $B(0,R)$ into a collection $\mathcal{Q}$ of squares of side length $K^2$ with bounded overlaps, and let $k$ range over dyadic numbers with $k\ge -C\log_2 K$. Define
\[
\mathcal Q_k := \{ Q \in \mathcal{Q} : \textstyle\int_Q H(x)\,dx \sim 2^{-k} \}.
\]
We may assume that $R^{-100} \le 2^{-k} \le R$, since the contribution from the range $2^{-k} < R^{-100}$ is negligible. By a dyadic pigeonholing argument, we may further assume that there exists a dyadic number $k_0$ such that
\begin{equation}\label{eq:weighted-L2}
    \|\mathcal{T}^\lambda f\|_{L^2(B(0,R); H\,dx)}^2 
    \le C_\varepsilon R^\varepsilon \sum_{Q \in \mathcal Q_{k_0}} \|\mathcal{T}^\lambda f\|_{L^2(Q;\,Hdx)}^2.
\end{equation}
Set
\[
\mu_Q =
\begin{cases}
1, & Q \in \mathcal Q_{k_0},\\[4pt]
0, & \text{otherwise}.
\end{cases}
\]
Then the right-hand side of \eqref{eq:weighted-L2} is bounded by
\begin{equation}\label{eq:muL2}
    2^{-k_0}\sum_{Q} \mu_Q^2 \|\mathcal{T}^\lambda f\|_{L^\infty(Q)}^2.
\end{equation}

To incorporate the ball condition for $H$ (see \eqref{fractal}), namely
\[
\int_{B_r } H(x)\,dx \leq A_\alpha(H)\, r^{\alpha},\qquad K^2 \leq r\leq R,
\]
we require that
\begin{equation}\label{eq:mucondition}
    \sum_{Q\subset B_r} \mu_{Q} \leq A_\alpha(H)\,2^{k_0} r^{\alpha}, 
    \qquad K^2 \le r \le R.
\end{equation}

Define
\begin{equation}\label{eq:defW}
    \mathcal{W}
    := \sup_{1 \le \rho  \le R^{1/2}}
    \rho^{-1/2} 
    \sup_{O \in \mathcal{O}}
    \Bigg( 
        \sum_{\substack{I \in \mathcal{I}\\ I \subset O}}
        \Big( \sum_{Q \subset I} \mu_Q^2 \Big)^2
    \Bigg)^{1/4}.
\end{equation}
Here, the summation $\sum_{I\in\mathcal I,\, I\subset O}$ runs over the subcollection of tubes $I\in\mathcal I$ that have the same direction as $O$ and form a finitely overlapping cover of $O$.

It is straightforward to show that 
\begin{equation}\label{eq:Wbound}
    \mathcal{W} \le C_\varepsilon\,A_\alpha(H)^\frac12\, 2^{\frac{k_0}{2}} R^{\tfrac{\alpha}{4}+\varepsilon}.
\end{equation}
Indeed, using that the tubes $I$ have bounded overlap, we obtain
\[
\begin{aligned}
    \mathcal{W}
    &= \sup_{1 \le \rho  \le R^{1/2}}
       \rho^{-1/2}
       \sup_{O \in \mathcal{O}}
       \Bigg( \sum_{\substack{I \in \mathcal{I}\\ I \subset O}} 
       \Big( \sum_{Q \subset I} \mu_Q^2 \Big)^2 \Bigg)^{1/4} \\
    &\le \sup_{1 \le \rho \le R^{1/2}}
       \rho^{-1/2}
       \sup_{O \in \mathcal{O}}
       \Big( \sum_{Q \subset O} \mu_Q^2 \Big)^{1/4}
       \sup_{I \in \mathcal{I}} \Big( \sum_{Q \subset I} \mu_Q^2 \Big)^{1/4} \\
    &\le C_\varepsilon A_\alpha(H)^{1/2} R^{\varepsilon}
       \sup_{1 \le \rho \le R^{1/2}}
       \rho^{-1/2}
       \Big( 2^{k_0}\, \rho (R/\rho)^{\alpha} \Big)^{1/4}
       \Big( 2^{k_0}\, \rho \cdot \rho^{\alpha} \Big)^{1/4} \\
    &\le C_\varepsilon A_\alpha(H)^{1/2} R^{\varepsilon} 
       2^{k_0/2} R^{\alpha/4}.
\end{aligned}
\]
Here, in the second-to-last inequality, we applied the ball condition \eqref{eq:mucondition} twice, with radii $r=R/\rho$ and $r=\rho$, respectively. Consequently, to prove \eqref{maines1}, it suffices to establish the following proposition, which is the key technical estimate in the proof of our main theorem.

\begin{proposition}\label{prop muQ}  For every $\varepsilon > 0$ and $1 \leq R \leq \lambda$, there exists a constant ${ C_\varepsilon}> 0$ such that
\begin{equation}\label{eq:muQ}
  \Big( \sum_{Q} \mu_Q^2 \|\mathcal T^\lambda f\|_{L^\infty(Q)}^2 \Big)^{\frac{1}{2}}
  \leq C_\varepsilon\, R^\varepsilon \, \mathcal{W}\, \|f\|_{L^2}.
\end{equation}
Here $\mathcal W$ is as in \eqref{eq:defW}.
\end{proposition}

In fact, using \eqref{eq:weighted-L2}, \eqref{eq:muL2}, \eqref{eq:Wbound} and \eqref{eq:muQ}, we obtain
\begin{align*}
    \int_{B(0,R)} |\mathcal{T}^\lambda f|^2 H(x)\,dx 
    &\le C_\varepsilon R^{3\varepsilon} 2^{-k_0}  \sum_{Q} \mu_{Q}^2 \| \mathcal{T}^\lambda f \|_{L^\infty(Q)}^2  \\
    &\le C_\varepsilon R^{3\varepsilon} A_\alpha(H)2^{-k_0} 2^{k_0} R^{\frac{\alpha}{2}} \|f\|_{L^2}^2.
\end{align*}

In general, let $\mu_Q \ge 0$ be arbitrary for each $Q \in \mathcal{Q}$. The definition of $\mathcal W$ is unchanged. Define $\mathcal{F}(\lambda, R)$ to be the smallest positive number such that 
\begin{equation*}
    \Bigg( 
        \sum_{Q} \mu_Q^2 
        \|\mathcal{T}^\lambda f\|_{L^\infty(Q)}^2
    \Bigg)^{1/2}
    \le 
    \mathcal{F}(\lambda, R)\, 
    \mathcal{W}\,
    \|f\|_{L^2}.
\end{equation*}
To prove \eqref{eq:muQ}, it suffices to show for each $\varepsilon>0$, there exists $\bar{C}_\varepsilon>0$ such that
\begin{equation}\label{inred}
    \mathcal{F}(\lambda,R)\leq \bar{C}_\varepsilon R^\varepsilon,\quad  1\leq R\leq \lambda. 
\end{equation}

\noindent \textbf{Induction hypothesis.} To prove \eqref{inred},  we argue by induction on scales. First, it is straightforward to verify that, for $N_\varepsilon\in\N$ to be chosen later, one can choose $\bar C_\varepsilon$ sufficiently large so that
\begin{equation}
\mathcal{F}(\lambda,R)\le \bar C_\varepsilon R^\varepsilon,\qquad \text{for all } 1\le R\le \min\{N_\varepsilon,\,\lambda\}.
\end{equation}
Now suppose $R\ge N_\varepsilon$. As the induction hypothesis, assume that
\begin{equation}\label{assumption}
\mathcal{F}(\bar\lambda,\bar R)\le \bar C_\varepsilon \bar R^\varepsilon,
\end{equation}
for all $1\le \bar R\le R/2$ and $\bar R\le \bar\lambda\le \lambda/2$. To close the induction, it suffices to establish
\begin{equation}\label{ingoal}
\mathcal{F}(\lambda,R)\le \bar C_\varepsilon R^\varepsilon,
\end{equation}
under the assumptions $R\ge N_\varepsilon$ and \eqref{assumption}. To achieve this, we invoke the following parabolic rescaling lemma, which links estimates across different scales.

\subsubsection{Parabolic rescaling lemma}
For technical reasons, to complete the induction argument in the narrow case, it is customary to prove estimates for a class of phase functions. 
To this end, we impose quantitative conditions on the phase function $\phi$ and introduce the notion of \emph{reduced form}.

Roughly speaking, a phase function $\phi(x,\xi)$ satisfying the Carleson--Sj\"olin conditions can be regarded as a small perturbation of the translation-invariant case. 
More precisely, through a suitable change of variables, $\phi$ can be written as
\begin{equation}\label{eq:09}
    \phi(x,\xi) = \langle x_1, \xi \rangle + x_2 h(\xi) + \mathcal{E}(x,\xi),
\end{equation}
where $h$ and $\mathcal{E}$ are smooth functions, $h$ is quadratic in $\xi$, and $\mathcal{E}$ is quadratic in both $x$ and $\xi$.\footnote{Here, “quadratic” means: for $(\mathbf a,\mathbf b)\in\mathbb{N}^2\times\mathbb{N}$,
\begin{itemize}
    \item $\partial_\xi^{\mathbf b} h(0)=0$ and $\partial_x^{\mathbf a}\partial_\xi^{\mathbf b}\mathcal{E}(x,0)=0$ for $\mathbf b\le 1$;
    \item $\partial_x^{\mathbf a}\partial_\xi^{\mathbf b}\mathcal{E}(0,\xi)=0$ for $|\mathbf a|\le 1$.
\end{itemize}

}

\begin{definition}
We say that a phase function $\phi(x,\xi)$ is in reduced form if the following conditions hold.
For $(x,\xi)\in B(0,1)\times[0,1]$, the phase $\phi$ has the form 
\begin{equation*}
    \phi(x,\xi)
    = x_1 \xi + x_2 h(\xi) + \mathcal{E}(x,\xi),
\end{equation*}
where $h$ and $\mathcal{E}$ are smooth functions. $h$ is quadratic in $\xi$, and $\mathcal{E}$ is quadratic in $(x,\xi)$. Moreover,
\begin{equation}\label{eq:uni}
    \bigl|\partial_x^{\mathbf a}\partial_\xi^{\mathbf b}\phi(x,\xi)\bigr|
    \le C_{\mathbf a,\mathbf b},
    \qquad |\mathbf a|,\mathbf b \le N_{\rm par},\ \mathbf a\in\mathbb{N}^2,\ \mathbf b\in\mathbb{N}.
\end{equation}
Here $N_{\rm par}$ is a fixed large constant.
Furthermore, $\phi$ satisfies the following:
\begin{itemize}
    \item[${\bf C_1}:$] $1/2 \le |h''(\xi)| \le 2$.
    \item[${\bf C_2}:$] Let $c_{\rm par}>0$ be a small fixed constant. We have
    \[
        \bigl|\partial_x^{\mathbf a}\partial_\xi^{\mathbf b}\mathcal{E}(x,\xi)\bigr| \le c_{\rm par},
        \quad |\mathbf a|,\mathbf b \le N_{\rm par},\ \mathbf a\in\mathbb{N}^2,\ \mathbf b\in\mathbb{N}.
    \]
\end{itemize}
\end{definition}

As in the discussion of (3.7) in \cite{GLW}, to establish \eqref{inred}, it suffices to consider the phase function in reduced form.

To derive the estimate for the narrow case, we require the following \emph{parabolic rescaling lemma}, which connects estimates across different scales. 
Such parabolic rescaling lemmas are widely used in multi-scale wave-packet analysis; 
see, for example, \cite{BHS,ILX,GLW} for similar variable-coefficient parabolic rescaling results. 
Here, we provide only an outline of the proof. For a detailed treatment, the reader is referred to the proof of inequality (3.7) in \cite{GLW}.  

We cover the interval $[0,1]$ by  $K$ open subintervals $\tau$ of length $K^{-1}$ with bounded overlap. 
Let $\{\tilde\psi_\tau\}$ be a smooth partition of unity subordinate to this cover. 
Correspondingly, define
\begin{equation*}
    f_\tau = f\,\tilde\psi_\tau, \qquad \mathcal{T}^\lambda f = \sum_\tau \mathcal{T}^\lambda f_\tau.
\end{equation*}
Recall that $T'_{\tau,v}$ are tubes in the collection $\mathfrak{T}'_\tau$. We have the following lemma.
\begin{lemma}[Parabolic rescaling lemma]\label{para lemma}
For each $T'_{\tau,v} \in \mathfrak{T}'_\tau$, we have
\begin{equation}\label{parabolic}
    \Bigg(\sum_{Q' \subset T'_{\tau,v}}\mu_{Q'}^2\,
    \|\mathcal T^\lambda f_\tau\|_{L^\infty(Q')}^2\Bigg)^{1/2}
    \lesssim \mathcal{F}(\lambda/K^2, R/K^2)\,K^{-1/2}\,\mathcal{W}'\,\|f_\tau\|_{L^2},
\end{equation}
where $\mathcal{W}'$ is defined by
\begin{equation*}
    \mathcal{W}':=
    \sup_{1\le \rho \le R_1^{1/2}}\rho^{-1/2}
    \sup_{O'\in \mathcal{O}'}
    \Bigg(\sum_{\substack{I' \in \mathcal{I}'\\ I' \subset O'}}
    \Big(\sum_{Q'\subset I'}\mu_{Q'}^2 \Big)^2\Bigg)^{1/4},
    \qquad R_1=R/K^2.
\end{equation*}
Here, the summation $\sum_{I' \in \mathcal{I'},\,I'\subset O'}$ is taken over a  subcollection of tubes in $\mathcal I'$ that have the same direction as $O'$ and form a finitely overlapping cover of $O'$.
\end{lemma}

\begin{remark}
For later use, we need a strengthened version of \eqref{parabolic}, namely
\begin{equation}\label{parabolic1}
    \Bigg(\sum_{Q' \subset T'_{\tau,v}}\mu_{Q'}^2\,
    \|\mathcal T^\lambda f_\tau\|_{L^\infty(Q')}^2\Bigg)^{1/2}
    \lesssim \mathcal{F}(\lambda/K^2, R/K^2)\,K^{-1/2}\,\mathcal{W}'\,\|f_{\tau,v}\|_{L^2},
\end{equation}
where
\[
    f_{\tau,v}
    = \big(\widehat{f_\tau}\,\tilde\eta_v\big)^{\vee},
\]
and $\tilde\eta_v$ is a Schwartz function essentially supported on an interval of length $R/K$ centered at $v$. Since we are considering $\mathcal{T}^\lambda f_\tau$ on the tube $T'_{\tau,v}$, it follows that 
\begin{equation}
    \mathcal{T}^\lambda f_\tau\big|_{T'_{\tau,v}}
    = \mathcal{T}^\lambda f_{\tau,v} + \mathrm{RapDec}(R)\,\|f\|_{L^2}.
\end{equation}
\end{remark}

\begin{proof}[Outline of proof for Lemma \ref{para lemma}.]
First, perform the change of variables 
\begin{equation*}
   (x_1,x_2)\mapsto \big(\gamma_\tau^\lambda(x_1,x_2),\,x_2\big)
\end{equation*}
so that
\begin{equation*}
    \partial_\xi \phi^\lambda\big(\gamma_\tau^\lambda(x_1,x_2),x_2,\xi_\tau\big)=x_1.
\end{equation*}
Under the new coordinates, $T_{\tau,v}'$ becomes a rectangle of dimensions $R/K\times R$ with sides parallel to the coordinate axes. Similarly, the squares $Q'$ become rectangles of dimensions $K K_1^2\times K^2 K_1^2$. Now consider the phase 
\begin{equation}
\phi^\lambda\big(\gamma_\tau^\lambda(x_1,x_2),x_2,\xi\big).
\end{equation}
Make the change of variables in frequency
\begin{equation*}
    \xi\mapsto K^{-1}\xi,
\end{equation*}
and correspondingly in space
\begin{equation*}
    x_1 \mapsto K x_1,\qquad x_2\mapsto K^2 x_2.
\end{equation*}

In the end, we reduce to a new phase $\tilde{\phi}^{\tilde{\lambda}}$ with $\tilde{\lambda}=\lambda/K^2$ such that
\begin{equation}\label{eq newph}
\phi^\lambda\big(\gamma_\tau^\lambda(Kx_1,K^2x_2),\,K^2x_2,\,K^{-1}\xi\big)
= \tilde{\phi}^{\tilde{\lambda}}(x_1,x_2,\xi).
\end{equation}
Here we set
\[
\tilde{\phi}(x_1,x_2,\xi):=K^2\,\phi\big(\gamma_\tau(K^{-1}x_1,x_2),\,x_2,\,K^{-1}\xi\big),
\]
so that $\tilde{\phi}^{\tilde{\lambda}}$ denotes the $\tilde{\lambda}$-rescaling of $\tilde{\phi}$ in the same sense that $\phi^\lambda$ denotes the $\lambda$-rescaling of $\phi$. 
By a standard change of variables (and a Taylor expansion), the new phase $\tilde{\phi}$ can be written in reduced form as 
\[
x_1\xi+x_2\,\tilde{h}(\xi)+\tilde{\mathcal{E}}(x,\xi),
\]
for some $\tilde{h},\tilde{\mathcal{E}}$ satisfying ${\bf C_1}$ and ${\bf C_2}$ with constants uniform in $\lambda$ and $K$. In the new coordinates, $T'_{\tau,v}$ and $Q'$ become squares of side length $R/K^2$ and $K_1^2$, respectively.

Finally, by the definition of $\mathcal{F}(\lambda/K^2, R/K^2)$, we obtain
\begin{equation*}
    \Bigg(\sum_{Q' \subset T'_{\tau,v}}\mu_{Q'}^2\,
    \|\mathcal T^\lambda f_\tau\|_{L^\infty(Q')}^2\Bigg)^{1/2}
    \lesssim \mathcal{F}(\lambda/K^2, R/K^2)\,K^{-1/2}\,\mathcal{W}'\,\|f_\tau\|_{L^2}.
\end{equation*}
Here the factor $K^{-1/2}$ appears as a consequence of the change of variables in frequency $\xi$.  
\end{proof}

\subsubsection{Bourgain--Guth argument}
Now we are ready to prove \eqref{inred}.

\begin{definition}
Given a square $Q\in\mathcal{Q}$ of side length $K^2$, define its \textbf{significant set} to be the set of length-$K^{-1}$ subintervals of $[0,1]$ given by
\[
\mathcal{S}(Q):=\left\{\tau\subset[0,1]\text{ with }|\tau|=K^{-1}\colon
\bigl\|\mathcal{T}^\lambda f_\tau\bigr\|_{L^\infty(Q)}\ge \frac{\bigl\|\mathcal{T}^\lambda f\bigr\|_{L^\infty(Q)}}{100K}
\right\}.
\]
We say that a $K^2$-square $Q$ is \textbf{broad} if there exist $\tau_1, \tau_2 \in \mathcal{S}(Q)$ such that
\begin{equation*}
    {\rm dist}(\tau_1,\tau_2) \ge 1000 K^{-1},
\end{equation*}
and \textbf{narrow} otherwise.
\end{definition}

It is straightforward to verify that, if $Q$ is narrow, then there exists some $\tau$ such that
\[
    \|\mathcal{T}^\lambda f_\tau\|_{L^\infty(Q)}
    \ge \frac{1}{10000}\,\|\mathcal{T}^\lambda f\|_{L^\infty(Q)}.
\]
In contrast, if $Q$ is broad, then there exist two significant caps $\tau_1,\tau_2$ such that the corresponding wave packets associated to $\tau_1$ and $\tau_2$ both pass through $Q$, and the angle between them is $\ge 1000 K^{-1}$. See Figure \ref{bn} for a comparison of narrow and broad squares $Q$.

Correspondingly, we decompose $B_R$ into squares $Q$ of side length $K^2$ with bounded overlap and define
\begin{equation*}
    \mathbf{B} := \bigcup_{\substack{Q\ \text{broad}}} Q, 
    \qquad 
    \mathbf{N} := \bigcup_{\substack{Q\ \text{narrow}}} Q.
\end{equation*}
Note that
\begin{equation*}
    \Bigg( \sum_{Q \subset B_R} \mu_Q^2 \,\|\mathcal{T}^\lambda f\|_{L^\infty(Q)}^2 \Bigg)^{1/2} 
    \le 
    \Bigg( \sum_{Q \subset \mathbf{B}} \mu_Q^2 \,\|\mathcal{T}^\lambda f\|_{L^\infty(Q)}^2 \Bigg)^{1/2}
    +
    \Bigg( \sum_{Q \subset \mathbf{N}} \mu_Q^2 \,\|\mathcal{T}^\lambda f\|_{L^\infty(Q)}^2 \Bigg)^{1/2}.
\end{equation*}
We apply the bilinear estimate \eqref{multi} to the broad part, and an induction-on-scales argument to the narrow part.

\begin{figure}
    \centering
    \begin{tikzpicture}[scale=1.2, line cap=round, line join=round]

\def\tubeh{0.28}   
\def\tubelen{2.2}  
\def\sqr{0.20}

\coordinate (Ncenter) at (-3,0);

\foreach \ang in {-6,0,6}{
  \begin{scope}[rotate around={\ang:(Ncenter)}]
    \draw[thick] 
      ($(Ncenter)+(-\tubelen,-\tubeh)$) 
      rectangle 
      ($(Ncenter)+(\tubelen,\tubeh)$);
  \end{scope}
}

\draw[thick,fill=cyan!30] 
  ($(Ncenter)+(-\sqr,-\sqr)$) rectangle ($(Ncenter)+(\sqr,\sqr)$);

\draw[dashed] (Ncenter) -- ++(-8:2.2);
\draw[dashed] (Ncenter) -- ++( 8:2.2);

\draw[->, shorten >=1pt] 
  (Ncenter) ++(-8:1.9) arc[start angle=-8, end angle=8, radius=1.9];

\path (Ncenter) ++(0:3.0) node {$\lesssim K^{-1}$};

\node[below] at ($(Ncenter)+(0,-2.0)$) {\small Narrow};

\coordinate (Bcenter) at (3,0);

\foreach \ang in {-30,30}{
  \begin{scope}[rotate around={\ang:(Bcenter)}]
    \draw[thick] 
      ($(Bcenter)+(-\tubelen,-\tubeh)$) 
      rectangle 
      ($(Bcenter)+(\tubelen,\tubeh)$);
  \end{scope}
}

\draw[thick,fill=orange!30] 
  ($(Bcenter)+(-\sqr,-\sqr)$) rectangle ($(Bcenter)+(\sqr,\sqr)$);

\draw[dashed] (Bcenter) -- ++(-30:1.8);
\draw[dashed] (Bcenter) -- ++( 30:1.8);

\draw[->, shorten >=1pt] 
  (Bcenter) ++(-30:1.5) arc[start angle=-30, end angle=30, radius=1.5];

\path (Bcenter) ++(0:2.4) node {$\gtrsim K^{-1}$};

\node[below] at ($(Bcenter)+(0,-2.0)$) {\small Broad};

\end{tikzpicture}

\caption{
Left (narrow): all wave-packets that contribute significantly point within an angular aperture $\lesssim K^{-1}$.
Right (broad): there exist two significantly contributing wave-packets whose directions are separated by $\gtrsim K^{-1}$. 
}

\label{bn}
\end{figure}

\medskip
{\noindent \bf Broad case.} First, we bound the Broad term.
Note that if $Q$ is broad, then there exist two intervals $\tau_1,\tau_2\in\mathcal S(Q)$ with 
$${\rm dist}(\tau_1,\tau_2)\geq 1000 K^{-1}.$$
Thus by the locally constant property,   we have 
\begin{align}\label{locala}
    \|\mathcal T^\lambda f\|_{L^\infty(Q)}\lesssim K^{C} \sum_{{\rm dist}(\tau_1,\tau_2 )\geq 1000 K^{-1}}\big\||\mathcal T^\lambda f_{\tau_1} \mathcal T^\lambda f_{\tau_2}|^{\frac{1}{2}}\big\|_{L^\infty(Q)}.
\end{align}
Applying the bilinear estimate 
Theorem \ref{multit} with $\nu=1000K^{-1}$ and $\varepsilon_1=\varepsilon/2$,  we obtain
\begin{align*}
    ( \sum_{Q\subset \mathbf B} \mu_Q^2 \|\mathcal T^\lambda f\|_{L^\infty(Q)}^2)^{1/2}
    &\lesssim K^C ( \sum_{Q\subset \mathbf B} \mu_Q^4 )^{1/4}(\sum_{Q\subset \mathbf B}\|\mathcal T^\lambda f\|_{L^\infty(Q)}^4)^{1/4}\\
    &\lesssim K^{C} (\sum_{Q\subset \mathbf B} \mu_Q^4)^{1/4}
  \!\!\! \!\!\!\!\!\!\sum_{{\rm dist}(\tau_1,\tau_2)>1000K^{-1}}(\sum_{Q\subset \mathbf B}\||\mathcal T^\lambda f_{\tau_1}\mathcal T^\lambda f_{\tau_2}|^{1/2}\|_{L^\infty(Q)}^4)^{1/4}\\
    &\lesssim K^C(\sum_{Q\subset \mathbf B}\mu_Q^4)^{1/4}
    \!\!\! \!\!\!\!\!\!\sum_{{\rm dist}(\tau_1,\tau_2)>1000K^{-1}} \||\mathcal T^\lambda f_{\tau_1}\mathcal T^\lambda f_{\tau_2}|^{1/2}\|_{L^4(B_R)}\\
    &\le C_{\varepsilon} K^C R^{\varepsilon/2} (\sum_{Q\subset \mathbf B} \mu_Q^4)^{1/4}\|f\|_{L^2}.
\end{align*}

In the second-to-last inequality, we bound each $L^\infty(Q)$-norm by its $L^4(Q)$-norm using the locally constant property (Lemma \ref{localc}), and then sum over $Q\subset \mathbf{B}$. The resulting loss is absorbed into the $K^{C}$ factor. Note that, choosing $\rho=1$ in the definition of $\mathcal W$, we obtain
\[
   \Big(\sum_{Q\subset \mathbf B} \mu_Q^4\Big)^{1/4}\le \mathcal{W}.
\]
Therefore
\begin{equation}\label{eq B}
    \Bigg( \sum_{Q \subset \mathbf{B}} \mu_Q^2
       \|\mathcal{T}^\lambda f\|_{L^\infty(Q)}^2
\Bigg)^{1/2}
\le C_\varepsilon K^{C} R^{\varepsilon/2} \, \mathcal W \, \|f\|_{L^2}.
\end{equation}

\medskip
\noindent\textbf{Narrow case.}
Now we estimate the narrow term. Note that if a square $Q$ is narrow, then there exists some $\tau$ such that 
\begin{equation*}
    10000\,\|\mathcal T^\lambda f_\tau\|_{L^\infty(Q)} \ge \|\mathcal T^\lambda f\|_{L^\infty(Q)}.
\end{equation*}
Thus, for every narrow square $Q$, we have
\begin{equation*}
    \|\mathcal T^\lambda f\|_{L^\infty(Q)} \lesssim \Big( \sum_\tau \|\mathcal T^\lambda f_\tau\|_{L^\infty(Q)}^2 \Big)^{1/2}.
\end{equation*}
Therefore,
\begin{equation}\label{narrow}
   \Big( \sum_{Q\subset \mathbf N} \mu_Q^2 \|\mathcal T^\lambda f\|_{L^\infty(Q)}^2 \Big)^{1/2}
   \lesssim 
   \Big( \sum_{\tau} \sum_{Q\subset \mathbf N} \mu_Q^2 \|\mathcal T^\lambda f_\tau\|_{L^\infty(Q)}^2 \Big)^{1/2}.
\end{equation}
We decompose $B_R$ into a collection of finitely overlapping curved tubes $T'_{\tau,v}$ such that
\begin{equation*}
    B_R \subset \bigcup_{v} T'_{\tau,v}.
\end{equation*}
It follows that 
\begin{align*}
 \Big( \sum_{Q\subset \mathbf N} \mu_Q^2 \|\mathcal T^\lambda f\|_{L^\infty(Q)}^2 \Big)^{1/2}
 &\lesssim 
 \Big( \sum_{\tau} \sum_{Q\subset B_R} \mu_Q^2 \|\mathcal T^\lambda f_\tau\|_{L^\infty(Q)}^2 \Big)^{1/2} \\
 &\lesssim 
 \Big( \sum_{\tau} \sum_{T'_{\tau,v}\subset B_R} 
 \sum_{Q\subset T'_{\tau,v}} \mu_Q^2 \|\mathcal T^\lambda f_\tau\|_{L^\infty(Q)}^2 \Big)^{1/2}.
\end{align*}
Noting that
\begin{align*}
  \sum_{Q\subset T'_{\tau,v}} \mu_Q^2 \|\mathcal T^\lambda f_\tau\|_{L^\infty(Q)}^2
  &\lesssim 
  \sum_{Q'\subset T'_{\tau,v}} \sum_{Q\subset Q'} \mu_Q^2 \|\mathcal T^\lambda f_\tau\|_{L^\infty(Q)}^2,
\end{align*}
we define
\begin{align}\label{muq}
    \mu_{Q'}^2 := \sum_{Q\subset Q'} \mu_Q^2,
\end{align}
and observe that
\begin{equation*}
    \sum_{Q\subset Q'} \mu_Q^2 \|\mathcal T^\lambda f_\tau\|_{L^\infty(Q)}^2 
    \lesssim  
    \mu_{Q'}^2 \|\mathcal T^\lambda f_\tau\|_{L^\infty(Q')}^2.
\end{equation*}
Hence,
\begin{equation*}
    \sum_{Q\subset T'_{\tau,v}} \mu_Q^2 \|\mathcal T^\lambda f_\tau\|_{L^\infty(Q)}^2
    \lesssim 
    \sum_{Q'\subset T'_{\tau,v}} \mu_{Q'}^2 \|\mathcal T^\lambda f_\tau\|_{L^\infty(Q')}^2 .
\end{equation*}
Applying \eqref{parabolic1} to the right-hand side yields
\begin{equation}
\Big( \sum_{Q\subset T'_{\tau,v}} \mu_{Q}^2
\|\mathcal T^\lambda f_\tau\|_{L^\infty(Q)}^2 \Big)^{1/2}
\lesssim   
\mathcal{F}(\lambda/K^2, R/K^2)\, K^{-1/2}\, \mathcal{W}' \|f_{\tau,v}\|_{L^2}.
\end{equation}
We claim that
\[
    K^{-1/2}\,\mathcal W' \lesssim \mathcal W.
\]
Indeed, by the definitions of $\mathcal W'$ and \eqref{muq}, we have
\[
    \mathcal W'
    = \sup_{1\le \rho \le R_1^{1/2}} \rho^{-1/2}
    \sup_{O'\in \mathcal O'}
    \Bigg(
        \sum_{\substack{I' \in \mathcal I' \\ I' \subset O'}}
        \Big( \sum_{Q \subset I'} \mu_{Q}^2 \Big)^{2}
    \Bigg)^{1/4},
    \qquad R_1 = R/K^2.
\]
Each $O' \in \mathcal O'$ at scale $\rho$ corresponds to some 
$O \in \mathcal O$ at scale $K\rho$, and similarly 
$I' \subset O'$ corresponds to $I \subset O$. 
Thus, the factor $\rho^{-1/2}$ in $\mathcal W'$ becomes 
$(K\rho)^{-1/2}$ when viewed in $\mathcal W$, 
which produces the extra $K^{-1/2}$ and confirms the claim.

Summing over $\tau$ and $T'_{\tau,v}$ and using almost orthogonality, we obtain
\begin{equation}\label{eq N}
    \Bigg( \sum_{Q \subset \mathbf{N}} \mu_Q^2 \,\|\mathcal{T}^\lambda f\|_{L^\infty(Q)}^2 \Bigg)^{1/2}
\lesssim \,
      \mathcal F(\lambda/K^2, R/K^2)\, \mathcal W \,\|f\|_{L^2}.
\end{equation}

\medskip
\noindent\textbf{Closing the induction.} By combining the contributions from \eqref{eq B} (broad) and \eqref{eq N} (narrow), we obtain the recursive bound
\begin{equation}\label{eq:recursiveF}
\mathcal F(\lambda,R) \le C_\varepsilon K^{C} R^{\varepsilon/2} + C\, \mathcal F(\lambda/K^2, R/K^2).
\end{equation}
Here $K=R^\kappa$ with $\kappa=\varepsilon^2$, and $C$ is an absolute constant. Applying the induction hypothesis \eqref{assumption} at the smaller scale,
\[
\mathcal F(\lambda/K^2, R/K^2) \le \bar C_\varepsilon (R/K^2)^\varepsilon,
\]
turns \eqref{eq:recursiveF} into
\[
\mathcal F(\lambda,R) \le C_\varepsilon K^{C} R^{\varepsilon/2} + C \bar C_\varepsilon (R/K^2)^{\varepsilon}.
\]
Hence,
\[
\mathcal F(\lambda,R) \le \big( C_\varepsilon K^{C} R^{-\varepsilon/2} + C \bar C_\varepsilon K^{-2\varepsilon} \big) R^{\varepsilon}.
\]
Since $K=R^{\kappa}$ with $\kappa=\varepsilon^2$ and $R\ge N_\varepsilon$, choosing $N_\varepsilon$ sufficiently large yields
\[
C_\varepsilon K^{C} R^{-\varepsilon/2} + C \bar C_\varepsilon K^{-2\varepsilon} \le \bar C_\varepsilon.
\]
This closes the induction and completes the proof of Propositions \ref{prop muQ} and \ref{prop 1,2}.

\subsection{The case $ 0 < \alpha < \frac{1}{2} $}

In this subsection, we prove Proposition \ref{theo11} in the range $ 0 < \alpha < \frac{1}{2} $. Indeed, we show that for every $\varepsilon>0$, there exists $C_\varepsilon>0$ such that 
\begin{equation}\label{eql2e}
    \|\mathcal{T}^\lambda f\|_{L^q(B(0,R);\,H \, dx)} \le C_\varepsilon A_\alpha(H)^\frac1q R^\varepsilon  \|f\|_{L^2}, \quad q\geq  2 .
\end{equation}
For convenience, set $ d\omega := H(x) \, dx $. Since, by the definition of $\mathcal T^\lambda$,
\begin{equation*}
    \|\mathcal{T}^\lambda f\|_{L^\infty(d\omega)} \leq C \|f\|_{L^2},
\end{equation*}
it follows that
\begin{equation}\label{eq:dis}
    \|\mathcal{T}^\lambda f\|_{L^q(d\omega)}^q = q \int_0^{C \|f\|_{L^2}} t^{q-1} \,\omega\{ x : |\mathcal{T}^\lambda f(x)| > t \} \, dt .
\end{equation}
Note that the set $ \{x : |\mathcal{T}^\lambda f(x)| > t\} $ is contained in
\begin{equation*}
    \{ ({\rm Re}\, \mathcal{T}^\lambda f)_+ > \tfrac{t}{4} \} \cup \{ ({\rm Re}\, \mathcal{T}^\lambda f)_- > \tfrac{t}{4} \} \cup \{ ({\rm Im}\, \mathcal{T}^\lambda f)_+ > \tfrac{t}{4} \} \cup \{ ({\rm Im}\, \mathcal{T}^\lambda f)_- > \tfrac{t}{4} \} .
\end{equation*}
We will consider the first set as an example, since the others can be handled in the same manner. Let
\begin{equation}
    G := \{ ({\rm Re}\, \mathcal{T}^\lambda f)_+ > \tfrac{t}{4} \} .
\end{equation}
Thus,
\begin{equation}\label{eq G}
    \frac{t}{4} \, \omega(G) \leq \int_G ({\rm Re}\, \mathcal{T}^\lambda f) \, d\omega \leq \left| \int \mathcal{T}^\lambda f(x) \chi_G(x) \, d\omega(x) \right|,
\end{equation}
where $ \chi_G $ denotes the characteristic function of $ G $. It follows that
\begin{multline}\label{eq:bound_integral}
    \left| \int \mathcal{T}^\lambda f(x) \chi_G(x) \, d\omega(x) \right|
    = \left| \int f(\xi) \int e^{i \phi^\lambda(x,\xi)} a^\lambda(x,\xi) \chi_G(x) \, d\omega(x) \, d\xi \right| \\
    \leq \left\| \int e^{i \phi^\lambda(x,\xi)} a^\lambda(x,\xi) \chi_G(x) \, d\omega(x) \right\|_{L^2} \|f\|_{L^2}.
\end{multline}
It suffices to consider
\begin{equation}\label{eq:dual}
    \left\| \int e^{i \phi^\lambda(x,\xi)} a^\lambda(x,\xi) \chi_G(x) \, d\omega(x) \right\|_{L^2}^2 .
\end{equation}
Expanding \eqref{eq:dual}, we obtain
\begin{equation}\label{eq:expan}
    \int \int \int e^{i (\phi^\lambda(x,\xi) - \phi^\lambda(y,\xi))} a^\lambda(x,\xi) a^\lambda(y,\xi) \, d\xi \, \chi_G(x) \, d\omega(x) \, \chi_G(y) \, d\omega(y) .
\end{equation}
By a standard stationary phase argument (see, e.g., \cite[Chap.\ 5]{soggeFIO}),
\begin{equation}\label{kernel bd}
    \left| \int e^{i (\phi^\lambda(x,\xi) - \phi^\lambda(y,\xi))} a^\lambda(x,\xi) a^\lambda(y,\xi) \, d\xi \right| \lesssim \frac{1}{(1 + |x - y|)^{1/2}} .
\end{equation}
Therefore, \eqref{eq:expan} is controlled by
\begin{equation*}
    \int \int \frac{1}{(1 + |x - y|)^{1/2}} \chi_G(x) \, d\omega(x) \, \chi_G(y) \, d\omega(y) .
\end{equation*}
We claim that, uniformly in $y$,
\begin{equation}\label{claima}
    \int \frac{1}{(1 + |x - y|)^{1/2}} \chi_G(x) \, d\omega(x) \lesssim  A_\alpha(H) .
\end{equation}
Combining \eqref{eq G}–\eqref{kernel bd} with the claim \eqref{claima}, we obtain
\begin{equation}
    t^2 \omega(G)^2 \lesssim  A_\alpha(H) \|f\|_{L^2}^2 \omega(G),
\end{equation}
which implies
\begin{equation*}
    \|\mathcal{T}^\lambda f\|_{L^q(d\omega)}^q \lesssim  A_\alpha(H) \|f\|_{L^2}^2 \int_0^{C \|f\|_{L^2}} t^{q - 3} \, dt \lesssim A_\alpha(H) \|f\|_{L^2}^q ,
\end{equation*}
for $q>2$. The case $q=2$ also follows since we allow an $R^\varepsilon$ loss. It remains to prove our claim \eqref{claima}.

By our assumption on $H$, for each ball $ B(x,r) $ with $ r \geq 1 $,
\begin{equation}\label{grow}
    \omega(B(x,r)) \leq  A_\alpha(H)\, r^\alpha .
\end{equation}
Using \eqref{grow}, we obtain
\begin{equation}
    \begin{aligned}
        \int \frac{1}{(1 + |x - y|)^{1/2}} \chi_G(x) \, d\omega(x)
        &\leq \int_{B(y,1)} \frac{1}{(1 + |x - y|)^{1/2}} \chi_G(x) \, d\omega(x) \\
        &\quad + \sum_{k \in \mathbb{Z}_+} \int_{B(y,2^k) \setminus B(y,2^{k-1})} \frac{1}{(1 + |x - y|)^{1/2}} \chi_G(x) \, d\omega(x) .
    \end{aligned}
\end{equation}
The first term is bounded by a constant multiple of $A_\alpha(H)$. For each $ k \in \mathbb{Z}_+ $,
\begin{equation*}
    \int_{B(y,2^k) \setminus B(y,2^{k-1})} \frac{1}{(1 + |x - y|)^{1/2}} \chi_G(x) \, d\omega(x)
    \lesssim  2^{-k/2} \,\omega(B(y,2^k)) \lesssim A_\alpha(H)\, 2^{k (\alpha - 1/2)} .
\end{equation*}
Since $ 0 < \alpha < \tfrac{1}{2} $, the series $ \sum_{k=1}^\infty 2^{k (\alpha - 1/2)} $ converges. Therefore, we obtain \eqref{claima}  and complete  the proof of \eqref{eql2e} in this range of $\alpha$.

\subsection{The case $\frac{1}{2}\leq\alpha<1$}
Finally, we prove Proposition \ref{theo11} for $\frac{1}{2}\leq\alpha<1$. This case is a direct consequence of Proposition \ref{theo22}, via H\"older's inequality. Indeed,
\begin{equation}
   \begin{aligned}
\int_{B(0,R)} |\mathcal T^\lambda f(x)|^2 H(x)\,dx
&\leq \Big(\int_{B(0,R)} |\mathcal T^\lambda f(x)|^{4\alpha} H(x)\,dx \Big)^{\frac{1}{2\alpha}}
\Big(\int_{B(0,R)} H(x)\,dx\Big)^{1-\frac{1}{2\alpha}}\\
&\lesssim \big(R^\varepsilon \|f\|_{L^2}^{4\alpha}\big)^{\frac{1}{2\alpha}}\, R^{\alpha-\frac{1}{2}}\\
&\lesssim R^{\alpha-\frac{1}{2}+\varepsilon}\,\|f\|_{L^2}^2.
\end{aligned}
 \end{equation}

\section{Proof of Proposition \ref{theo22}}\label{sec 4}
The proof is analogous to that of Theorem 2.1 in \cite{shayya1}, except that our estimates also extend to the endpoint since we allow $\lambda^\varepsilon$ losses. 
For any Lebesgue measurable function $F:\mathbb{R}^2\to[0,\infty)$, define
\begin{equation*}
    M_\alpha F:=\Bigg(\sup_H \frac{1}{A_\alpha(H)}\int F(x)^\alpha H(x)\,dx\Bigg)^{1/\alpha},
\end{equation*}
where the supremum is taken over all nonzero weights $H$ on $\mathbb{R}^2$ of dimension $\alpha$. The following H\"older-type inequality is a key ingredient in the proof of Proposition \ref{theo22}.

\begin{theorem}[\cite{shayya1}]\label{theo33}
Suppose $n\ge 1$ and $0<\beta<\alpha\le n$. Then
\begin{equation*}
    M_\alpha F\le M_\beta F,
\end{equation*}
for all nonnegative Lebesgue measurable functions $F$ on $\mathbb{R}^n$.
\end{theorem}

We define $\mathfrak Q(\alpha,p)$ to be the infimum of all exponents $q$ satisfying the following property: for every $\varepsilon>0$, there exists $C_\varepsilon>0$ such that
\begin{equation}\label{qpestimate}
    \int_{B(0,R)} |\mathcal{T}^\lambda f(x)|^q H(x)\,dx 
    \le C_\varepsilon\,A_\alpha(H)\,R^{\varepsilon}\,\|f\|_{L^p}^q .
\end{equation}
Equivalently,
\begin{equation*}
    \mathfrak Q(\alpha,p)
    := \inf\Big\{ q : \forall\,\varepsilon>0 ,\exists\,C_\varepsilon>0 \text{ such that \eqref{qpestimate} holds}\Big\}.
\end{equation*}
We first establish a monotonicity property of $\mathfrak Q(\alpha,p)$ with respect to $\alpha$: for $0<\beta<\alpha\le 2$,
\begin{equation}\label{rela}
    \frac{\mathfrak Q(\alpha,p)}{\alpha}\le \frac{\mathfrak Q(\beta,p)}{\beta}.
\end{equation}
Indeed, fix $\varepsilon>0$ and suppose $q>\mathfrak Q(\beta,p)$. By the definition of $\mathfrak Q(\beta,p)$, there exists $C_\varepsilon>0$ such that
\begin{equation*}
    \int_{B(0,R)} |\mathcal T^\lambda f(x)|^q H(x)\,dx
    \le C_\varepsilon\,R^\varepsilon\,A_\beta(H)\,\|f\|_{L^p}^q ,
\end{equation*}
for all $f\in L^p$ and all weights $H$ of dimension $\beta$.
Set 
\[
    F = \chi_{B(0,R)}\,|\mathcal T^\lambda f|^{q/\beta}.
\]
Then it follows that
\begin{equation*}
    M_\beta(F)\le (C_\varepsilon R^\varepsilon \|f\|_{L^p}^q)^{1/\beta}.
\end{equation*}
Applying Theorem \ref{theo33}, we obtain
\begin{equation*}
    M_\alpha(F)\le (C_\varepsilon R^\varepsilon \|f\|_{L^p}^q)^{1/\beta}.
\end{equation*}
Hence,
\begin{equation*}
    \bigg(\frac{1}{A_\alpha(H)}\int_{B(0,R)} 
    |\mathcal T^\lambda f(x)|^{\frac{\alpha q}{\beta}} H(x)\,dx\bigg)^{1/\alpha}
    \le (C_\varepsilon R^\varepsilon \|f\|_{L^p}^q)^{1/\beta},
\end{equation*}
for all $f\in L^p$ and all weights $H$ of dimension $\alpha$.  
By the definition of $\mathfrak Q(\alpha,p)$, this implies
\[
    \frac{\alpha q}{\beta}\ge \mathfrak Q(\alpha,p).
\]
Since this holds for all $q>\mathfrak Q(\beta,p)$, we conclude that
\[
    \frac{\mathfrak Q(\beta,p)}{\beta}\ge \frac{\mathfrak Q(\alpha,p)}{\alpha},
\]
which proves \eqref{rela}.

\medskip
Now we are ready to prove Proposition \ref{theo22}. The case $0<\alpha\le \tfrac{1}{2}$ was handled in Proposition \ref{theo11}. See the proof of \eqref{eql2e}. 
As a direct consequence of \eqref{rela}, for $\tfrac{1}{2}<\alpha\le 1$ and $q>4\alpha$, we have
\begin{equation}
    \int_{B(0,R)} |\mathcal T^\lambda f(x)|^q H(x)\,dx
    \le C_\varepsilon\, R^\varepsilon\, A_\alpha(H)\, \|f\|_{L^2}^q .
\end{equation}
Indeed, by \eqref{eql2e} we have
\[
\frac{\mathfrak Q(\beta,2)}{\beta}=\frac{2}{\beta},\qquad 0<\beta<\tfrac12.
\]
Thus, by \eqref{rela},
\begin{equation*}
    \frac{\mathfrak Q(\alpha,2)}{\alpha}\le \inf_{\beta\in(0,\tfrac12)}\frac{2}{\beta}=4,
\end{equation*}
and the desired bound follows.

Finally, we consider the case $1<\alpha\le 2$.  
By Proposition \ref{prop 1,2}, we have
\begin{equation*}
    \int_{B(0,R)} |\mathcal T^\lambda f(x)|^2 H(x)\,dx
    \le C_\varepsilon\, R^\varepsilon\, A_\alpha(H)\, R^{\tfrac{\alpha}{2}} \|f\|_{L^2}^2.
\end{equation*}
Define
\begin{equation*}
    \bar H(x):=\|f\|_{L^2}^{-2}\,|\mathcal T^\lambda f(x)|^2\,H(x).
\end{equation*}
Then $\bar H$ is a weight of dimension $\bar\alpha:=\tfrac{\alpha}{2}+\varepsilon$, and moreover
\begin{equation*}
    A_{\bar\alpha}(\bar H)\lesssim A_\alpha(H).
\end{equation*}
Since $1<\alpha<2$, we have $1/2<\bar\alpha<1$ provided $\varepsilon>0$ is sufficiently small.  
Thus, by the result for the range $1/2<\bar\alpha<1$, 
\begin{equation*}
    \mathfrak Q(\bar\alpha,2)\le 4\bar\alpha.
\end{equation*}
Consequently, for all $g\in L^2$ and all $q'>4\bar\alpha=2\alpha+4\varepsilon$,
\begin{equation*}
    \int_{B(0,R)} |\mathcal T^\lambda g(x)|^{q'} \bar H(x)\,dx
    \le C_\varepsilon\, R^\varepsilon\, A_{\bar\alpha}(\bar H)\, \|g\|_{L^2}^{q'}.
\end{equation*}
Replacing $\bar H$ by $\|f\|_{L^2}^{-2}|\mathcal T^\lambda f|^2 H$, setting $g=f$, and choosing $\varepsilon>0$ sufficiently small, we obtain
\begin{equation*}
    \int_{B(0,R)} |\mathcal T^\lambda f(x)|^q H(x)\,dx
    \le C_\varepsilon\, R^\varepsilon\, A_\alpha(H)\, \|f\|_{L^2}^q,
\end{equation*}
for all $q>2\alpha+2$.  
This completes the proof.

\section{Proof of Theorem \ref{theo:cur}}\label{sec 5}
In this section we prove Theorem \ref{theo:cur} and explain its connection to the Mizohata--Takeuchi conjecture. As in the proof of Theorem \ref{theomain}, Theorem \ref{theo:cur} reduces to a bound for $\mathcal T^\lambda f$. Specifically, we estimate $\mathcal T^\lambda f$ on suitably rescaled tubular neighborhoods of a curve $\gamma$ tangent to $\{\Gamma_{\xi_\theta,v}\}_{\theta,v}$ to order at most $1$. See Definition \ref{order} below. We denote by $T_{\lambda^{s}}(\gamma^\lambda)$ the $\lambda^{s}$-neighborhood of the rescaled curve $\gamma^\lambda(\cdot) := \lambda \gamma(\cdot/\lambda)$. Theorem \ref{theo:cur} is a consequence of the following proposition.

\begin{proposition}\label{theo2}
Let $\gamma$ be a curve tangent to the family $\{\Gamma_{\xi_\theta,v}\}_{\theta,v}$ to order at most $1$, and let $0 \le s \le \tfrac{1}{2}$. For every $\varepsilon>0$ there exists $C_\varepsilon>0$ such that
\[
\|\mathcal T^\lambda f\|_{L^2\left(T_{\lambda^{s}}(\gamma^\lambda)\right)} \le C_\varepsilon \lambda^{\frac{1+s}{4}+\varepsilon} \|f\|_{L^2}.
\]
\end{proposition}
 
\begin{remark}
This result is sharp for $\tfrac{1}{3} \le s \le \tfrac{1}{2}$, but not for $0 < s < \tfrac{1}{3}$. For $0 < s < \tfrac{1}{3}$, the sharp bound follows from
\begin{equation}\label{eq:curs}
\| \mathcal T^\lambda f \|_{L^2(\gamma^\lambda)} \le C \lambda^{\tfrac{1}{6}} \| f \|_{L^2}.
\end{equation}
A representative example is
\begin{equation}\label{eq:rep}
    \int e^{i\lambda(t^2 \xi + t \xi^2)} a(t,\xi) f(\xi)\, d\xi, \quad (t,\xi) \in \mathbb{R} \times \mathbb{R},
\end{equation}
where $a \in C_c^\infty([0,1] \times [0,1])$. For \eqref{eq:rep}, take
\[
\phi(x,\xi) = x_1 \xi + x_2 \xi^2,
\]
and restrict to the curve $(x_1(t),x_2(t)) = \gamma(t) = (t^2,t)$. The estimate \eqref{eq:curs} can be deduced from a more general estimate. The classical result of Phong and Stein reads \cite{PSA}:
\begin{equation}
    \left\| \int e^{i\lambda\left( \sum_{\alpha} a_\alpha t^{\alpha_1} \xi^{\alpha_2} \right)} a(t,\xi) f(\xi)\, d\xi \right\|_{L^2} \lesssim  \lambda^{-\frac{1}{2 d_{\mathrm N}}} \| f \|_{L^2}, \quad \alpha=(\alpha_1,\alpha_2) \in \mathbb{N}^2,
\end{equation}
where $d_{\mathrm N}$ is the Newton distance defined by
\begin{equation*}
    d_{\mathrm N} := \inf \{ d \in \mathbb{R} : (d,d) \in \mathrm{NP} \},
\end{equation*}
and the Newton polygon $\mathrm{NP}$ is defined as
\begin{equation*}
    \mathrm{NP} := \text{convex hull of } \left\{ \alpha + \mathbb{R}_{+}^2 : a_\alpha \neq 0 \right\}.
\end{equation*}
For \eqref{eq:rep}, we have $d_{\mathrm N}=\tfrac{3}{2}$. By a further scaling argument in $t$, we obtain \eqref{eq:curs}. 
\end{remark}

Our argument for proving Proposition \ref{theo2} combines the wave-packet decomposition from the previous sections with a simple geometric observation. We begin with a geometric lemma. In $\R^2$, the $\delta$-neighborhoods of a line and a parabola cannot overlap too much. We generalize this observation to the Riemannian setting.

\subsection{A geometric lemma}
Let $\gamma\subset \mathbb{R}^2$ be a curve segment parametrized by arc length with nonvanishing curvature, that is,
\begin{equation*}
    |\gamma''(s)|\geq c,\quad s\in (0,\varepsilon_0),
\end{equation*}
for some fixed $0<\varepsilon_0\ll 1$ and constant $c>0$.

Let $0<\delta\ll 1$, and write $\gamma_\delta$ for the $\delta$-neighborhood of $\gamma$. Similarly, for a line $\ell$, write $\ell_\delta$ for the $\delta$-neighborhood of $\ell$. Whenever $\gamma_\delta$ intersects $\ell_\delta$, one has
\begin{equation*}
    \Big|\ell_\delta\cap \gamma_\delta\Big|\le C\,\delta^{3/2},
\end{equation*}
with equality in the extremal case when $\ell$ is tangent to $\gamma$. This fact extends to Riemannian manifolds, with the line $\ell$ replaced by a geodesic and $\gamma$ a curve with nonvanishing geodesic curvature.

\begin{lemma}\label{le41}
Let $\gamma_1$ be a curve segment in the Riemannian manifold $(M,g)$ with nonvanishing geodesic curvature, and let $\gamma_2$ be a geodesic segment. Let $(\gamma_1)_\delta$ and $(\gamma_2)_\delta$ denote the $\delta$-neighborhoods of $\gamma_1$ and $\gamma_2$, respectively. Then
\begin{equation}
    \Big|(\gamma_1)_\delta\cap (\gamma_2)_\delta\Big|\lesssim \delta^{3/2}.
\end{equation}
\end{lemma}

To prove Lemma \ref{le41} we first record some basic geometric lemmas.

\begin{lemma}\label{le21}
Let $(M,g)$ be a Riemannian manifold, and let $\exp_p:T_pM\to M$ be the exponential map at a point $p\in M$. Let $u:[0,\varepsilon)\to T_pM$ be a smooth curve with $u(0)=0$ and $|u'(0)|=1$. Define a curve $\gamma:[0,\varepsilon)\to M$ by
\begin{equation}
    \gamma(s)=\exp_p\big(u(s)\big).
\end{equation}
Assume that
\begin{equation}\label{eq:geo}
   \big( \nabla_{\gamma'(s)}\gamma'(s),\, \nabla_{\gamma'(s)}\gamma'(s) \big) \ge c>0,\quad s\in(0,\varepsilon),
\end{equation}
where $\nabla$ is the Levi-Civita connection of $g$. Then there exist $\tilde c>0$ and $0<\tilde\varepsilon<\varepsilon$ such that
\begin{equation}
    |u''(s)|\ge \tilde c,\quad s\in(0,\tilde\varepsilon).
\end{equation}
\end{lemma}

\begin{proof}
Work in geodesic normal coordinates centered at $p$, so that $g_{ij}(p)=\delta_{ij}$ and the Christoffel symbols satisfy $\Gamma_{ij}^k(p)=0$. Write
\begin{equation*}
    u(s)=\big(u_1(s),\dots,u_n(s)\big).
\end{equation*}
In these coordinates,
\begin{equation*}
    \nabla_{\gamma'(s)}\gamma'(s)
    = u''(s)+\Gamma_{ij}^k\big(u(s)\big)\,u'_i(s)\,u'_j(s)\,e_k,
\end{equation*}
where $\{e_k\}$ is the coordinate frame. Since $\Gamma_{ij}^k(p)=0$, by shrinking $\tilde\varepsilon$ if necessary we may assume
\begin{equation*}
    \big|\Gamma_{ij}^k\big(u(s)\big)\big|\le \frac{c}{1000},\qquad s\in(0,\tilde\varepsilon),
\end{equation*}
and, by continuity of $u'$, that $|u'(s)|\le 2$ on $(0,\tilde\varepsilon)$. Then, using \eqref{eq:geo} and the triangle inequality,
\begin{equation*}
    |u''(s)|
    \ge \big|\nabla_{\gamma'(s)}\gamma'(s)\big|
       - \big|\Gamma_{ij}^k\big(u(s)\big)\,u'_i(s)\,u'_j(s)\big|
    \ge \tilde c,
\end{equation*}
for some $\tilde c>0$, as claimed.
\end{proof}

\begin{proof}[Proof of Lemma \ref{le41}]
Assume $\gamma_1$ and $\gamma_2$ intersect at a point $p$. Work in geodesic normal coordinates centered at $p=\gamma_1(t_0)=\gamma_2(\tilde t_0)$. In these coordinates the geodesic $\gamma_2$ is represented by a straight line $\ell$, and we denote by $\gamma$ the image of $\gamma_1$.

By Lemma \ref{le21} there exist $\tilde c>0$ and $\tilde\varepsilon>0$ such that the Euclidean acceleration of $\gamma$ satisfies
\[
    |\gamma''(t)|\ge \tilde c,\qquad t\in(t_0-\tilde\varepsilon,\,t_0+\tilde\varepsilon).
\]
Equivalently, the Euclidean curvature of $\gamma$ is bounded below on this interval (after possibly shrinking $\tilde\varepsilon$).

In $\R^2$, as discussed above, for $0<\delta\ll1$, the overlap of the $\delta$–neighborhoods of a straight line and a  curve with curvature bounded below by a positive constant obeys
\[
    \Big|\ell_\delta\cap \gamma_\delta\Big|\lesssim \delta^{3/2},
\]
where the implicit constant depends only on the lower curvature bound. Hence, in our coordinates,
\begin{equation}\label{eq:34}
    \Big|(\gamma_2)_\delta\cap (\gamma_1)_\delta\Big|\lesssim \delta^{3/2}.
\end{equation}

To transfer \eqref{eq:34} to $(M,g)$, fix $r>0$ small so that on the normal coordinate ball $B_r$ one has
\[
    |g_{ij}(x)-\delta_{ij}|\le C r^2,\qquad x\in B_r,
\]
and the Riemannian volume element is comparable to the Euclidean one. For $\delta\ll r$, the sets $(\gamma_1)_\delta$ and $(\gamma_2)_\delta$ lie inside $B_r$, so the Jacobian and metric comparability give that Riemannian and Euclidean areas differ by at most a fixed multiplicative constant. Therefore the bound \eqref{eq:34} holds (up to a harmless constant) in the Riemannian metric as well. This completes the proof.
\end{proof}

\subsection{Proof of Proposition \ref{theo2} and its connection to the Mizohata--Takeuchi conjecture}
Before proving Proposition \ref{theo2}, we provide additional motivation by discussing the following conjecture of Mizohata--Takeuchi in the plane.
\begin{conjecture}[Mizohata--Takeuchi]\label{M-T}
Let $\Sigma$ be a $C^2$ curve in $\mathbb{R}^2$ with nonvanishing curvature, and let $d\mu_\Sigma$ be the measure on $\Sigma$ induced by the Lebesgue measure in $\mathbb{R}^2$. Define
\begin{equation*}
    \widehat{g\,d\mu_\Sigma}(x):=\int_\Sigma e^{-i x \cdot \xi}\, g(\xi)\, d\mu_\Sigma(\xi).
\end{equation*}
For any weight function $w:\mathbb{R}^2\to[0,\infty)$, set
\begin{equation*}
    \|Xw\|_{L^\infty}:=\sup_{\ell}\int_{\ell} w(x)\,dx,
\end{equation*}
where the supremum is taken over all lines $\ell\subset\mathbb{R}^2$. Then there exists a constant $C$, depending only on $\Sigma$, such that
\begin{equation}
    \int_{\mathbb{R}^2} \big|\widehat{g\,d\mu_\Sigma}(x)\big|^2 w(x)\,dx \le C\,\|Xw\|_{L^\infty}\int_\Sigma |g(\xi)|^2\, d\mu_\Sigma(\xi).
\end{equation}
\end{conjecture}

To relate the Mizohata--Takeuchi conjecture to Theorem \ref{theo:cur}, we propose the following variable-coefficient analogue.

\begin{conjecture}[Variable-Coefficient Mizohata--Takeuchi Conjecture]\label{VTM}
Let $\mathcal T^\lambda$ be as defined above. For any weight $w:B_\lambda\to[0,\infty)$, define
\begin{equation*}
    \|X^\lambda w\|_{L^\infty}:=\sup_{\theta,v}\int_{\Gamma_{\xi_\theta,v}^\lambda} w(x)\,dx,
\end{equation*}
where the supremum is taken over all curves in the collection $\{\Gamma_{\xi_\theta,v}^\lambda\}_{\theta,v}$. Then, for any $\varepsilon>0$,
\begin{equation}\label{vmt}
    \int_{B_\lambda} |\mathcal T^\lambda f(x)|^2 w(x)\,dx \le C_\varepsilon\,\lambda^\varepsilon\,\|X^\lambda w\|_{L^\infty} \int_{\mathbb{R}} |f(\xi)|^2\,d\xi,
\end{equation}
for some constant $C_\varepsilon>0$ independent of the choice of $w$.
\end{conjecture}

To see the consequence of Conjecture \ref{VTM} for our problem, we introduce the following notion.
\begin{definition}\label{order}
    We say that a curve $\gamma \subset B(0,1)$, parametrized by arc length, is \textit{transversal} to the family $\{\Gamma_{\xi_\theta,v}\}_{\theta,v}$ if there exists $c>0$ such that
\begin{equation}\label{eq:4}
\big|\,\partial_{x\xi}^2\phi\big(\gamma(t),\xi\big)\cdot\gamma'(t)\,\big|\ge c,
\quad \text{for all } t\in(0,\varepsilon),\ \xi\in(0,\varepsilon),
\end{equation}
where $\partial_{x\xi}^2\phi$ denotes the $x$-gradient of $\partial_\xi \phi$.

Given $\theta$ and $v$, we say that $\gamma$ is \textit{tangent} to $\Gamma_{\xi_\theta,v}$ to order $k$ if, after arc-length parametrization, the two curves meet at some parameter value $t_0$. Their derivatives agree up to order $k$ at $t_0$, and their $(k+1)$-st derivatives differ. For consistency, we regard the transversal case as tangency of order zero.

We say that $ \gamma $ is tangent to the family $ \{\Gamma_{\xi_\theta,v}\}_{\theta,v}$ to order at most $ k $ if, for every $ \Gamma_{\xi_\theta,v} $ in the family, $ \gamma $ is tangent to $ \Gamma_{\xi_\theta,v} $ to order at most $ k $, and there exists a uniform constant $ c>0 $ such that, at each intersection point, if their derivatives up to order $ k $ agree, then their $ (k+1) $-st derivatives are separated by at least $ c $.
\end{definition}

After rescaling, we write $T_{\lambda^{s}}(\gamma^\lambda)$ for the $\lambda^s$-neighborhood of $\gamma^\lambda$. Assuming Conjecture \ref{VTM}, we obtain
\begin{equation*}
    \int_{T_{\lambda^{s}}(\gamma^\lambda)} |\mathcal T^\lambda f(x)|^2\,dx
    \le C_\varepsilon \lambda^\varepsilon
    \sup_{\theta,v}\big|\Gamma_{\xi_\theta,v}^\lambda \cap T_{\lambda^{s}}(\gamma^\lambda)\big|\,\|f\|_{L^2}^2,
\end{equation*}
where $|\cdot|$ denotes arc length.

Since $\gamma$ is tangent to the family $\{\Gamma_{\xi_\theta,v}\}_{\theta,v}$ to order at most $k$, we have uniformly in $(\theta,v)$ that
\begin{equation*}
   \big|\Gamma_{\xi_\theta,v}^\lambda \cap T_{\lambda^{s}}(\gamma^\lambda)\big|
   \le \lambda^{\frac{k+s}{k+1}}.
\end{equation*}
Thus, for $0\le s\le \tfrac12$,
\begin{equation}\label{eq:mt1}
    \int_{T_{\lambda^{s}}(\gamma^\lambda)} |\mathcal T^\lambda f(x)|^2\,dx
    \le C_\varepsilon\,\lambda^{\frac{k+s}{k+1}+\varepsilon}\,\|f\|_{L^2}^2.
\end{equation}
In particular, if $k=1$, then
\begin{equation}\label{eq:mt2}
    \int_{T_{\lambda^{s}}(\gamma^\lambda)} |\mathcal T^\lambda f(x)|^2\,dx
    \le C_\varepsilon\,\lambda^{\frac{1+s}{2}+\varepsilon}\,\|f\|_{L^2}^2.
\end{equation}

We now prove that for $0 \le s \le \tfrac12$ the estimate predicted by the variable-coefficient Mizohata--Takeuchi conjecture holds. We will show in the final section that, when $k=1$, it is sharp for $\tfrac13 \le s \le \tfrac12$.

\begin{proposition}\label{theo4}
Let $0\le s\le\frac12.$ If $\gamma$ is tangent to the family $\{\Gamma_{\xi_\theta,v}\}$ to order at most $k\ge1$, then for any $\varepsilon>0$,
\begin{equation}
    \|\mathcal T^\lambda f\|_{L^2\big(T_{\lambda^{s}}(\gamma^\lambda)\big)}^2
    \le C_\varepsilon\,\lambda^{\frac{k+s}{k+1}+\varepsilon}\,\|f\|_{L^2}^2.
\end{equation}
\end{proposition}

{Proposition \ref{theo2} corresponds to the case $k=1$ in the above proposition, and Theorem \ref{theo:cur} is a direct consequence of Proposition \ref{theo2}, \eqref{eq refor}, and Lemma \ref{le41}. 
Moreover, Proposition \ref{theo4} may be viewed as a variable-coefficient generalization of Corollary~3.4 in \cite{CIW}, which corresponds to the case $k=0$.
}

\begin{proof}
Cover $T_{\lambda^{s}}(\gamma^\lambda)$ by a finitely overlapping family $\mathcal{B}$ of balls $B$ of radius $\lambda^{\frac{k+s}{k+1}}$. Then
\[
\int_{T_{\lambda^{s}}(\gamma^\lambda)} |\mathcal T^\lambda f(x)|^2\,dx
\le \sum_{B\in\mathcal{B}} \int_{B\cap T_{\lambda^{s}}(\gamma^\lambda)} |\mathcal T^\lambda f(x)|^2\,dx.
\]
By the wave-packet decomposition at scale $R=\lambda$ and Lemma \ref{lem tube}, only packets whose tubes meet $B$ contribute, up to a rapidly decaying error. Therefore,
\[
\int_{B\cap T_{\lambda^{s}}(\gamma^\lambda)} \big|\mathcal T^\lambda f\big|^2
\lesssim
\sum_{(\theta,v):\, T_{\theta,v}\cap B\cap T_{\lambda^{s}}(\gamma^\lambda)\neq\varnothing}
\int_{B} \big|\mathcal T^\lambda f_{\theta,v}\big|^2
+ \mathrm{RapDec}(\lambda)\,\|f\|_{L^2}^2.
\]
Here we use the local almost orthogonality of the packets 
$\mathcal{T}^\lambda f_{\theta,v}$ on balls $B$ of radius 
$\lambda^{\frac{k+s}{k+1}} \ge \lambda^{1/2}$, 
which can be easily seen by expanding the term 
$|\mathcal{T}^\lambda f|^2$ in terms of the wave packet decomposition, 
where the cross terms contribute a rapidly decaying error to the integral on $B$.

Applying H\"ormander’s $L^2$ bound on balls (see, for example, Section 2.1 of \cite{soggeFIO}),
\[
\int_{B} |\mathcal T^\lambda f_{\theta,v}(x)|^2\,dx \lesssim \lambda^{\frac{k+s}{k+1}}\,\|f_{\theta,v}\|_{L^2}^2,
\]
we obtain
\[
\int_{T_{\lambda^{s}}(\gamma^\lambda)} |\mathcal T^\lambda f(x)|^2\,dx
\lesssim \lambda^{\frac{k+s}{k+1}} \sum_{B\in\mathcal{B}} \sum_{(\theta,v):\, T_{\theta,v}\cap B\cap T_{\lambda^{s}}(\gamma^\lambda)\neq\varnothing}
\|f_{\theta,v}\|_{L^2}^2.
\]
By our geometric observation, we see that
\begin{equation}\label{eq hor}
\#\{B\in\mathcal{B}: T_{\theta,v}\cap B\cap T_{\lambda^{s}}(\gamma^\lambda)\neq\varnothing\}\lesssim \lambda^\varepsilon.
\end{equation}
Invoking the almost-orthogonality lemma \ref{ortho}, we conclude that
\[
\int_{T_{\lambda^{s}}(\gamma^\lambda)} |\mathcal T^\lambda f|^2
\lesssim \lambda^{\frac{k+s}{k+1}}\,\lambda^\varepsilon \sum_{\theta,v}\|f_{\theta,v}\|_{L^2}^2
\lesssim \lambda^{\frac{k+s}{k+1}+\varepsilon}\,\|f\|_{L^2}^2,
\]
as claimed.
\end{proof}

For $s<\tfrac{1}{3}$ and $k=1$, the Mizohata--Takeuchi conjecture yields a nonsharp estimate in \eqref{eq:mt1}. Indeed, if we localize $\mathcal T^\lambda f$ to a ball of radius $\lambda^{2s}$ and perform the wave-packet decomposition at that scale, then each wave packet $T_{\theta,v}$ can intersect $T_{\lambda^{s}}(\gamma^\lambda)$ in a segment of length $\lambda^{\frac{1+s}{2}}$. To obtain an improvement from the Mizohata--Takeuchi conjecture, the intersection length $\lambda^{(1+s)/2}$ must be dominated by the packet scale $\lambda^{2s}$, i.e.
\[
\frac{1+s}{2}<2s,
\]
which is equivalent to $s>\tfrac{1}{3}$.

For $0<s<\tfrac{1}{3}$, the sharp bound instead follows directly from the theorem of Pan and Sogge \cite{PS}, which quantifies the optimal gain for oscillatory integral operators with at most fold singularities. Similarly, this is exactly how Hu \cite{huforum} proved his version of Theorem \ref{bgt}.

\begin{remark}
We note a connection between the variable-coefficient Mizohata--Takeuchi conjecture and Kakeya--Nikodym estimates for eigenfunctions. From the Kakeya--Nikodym estimate of Blair and Sogge \cite{BS} one has
\begin{equation}\label{eq:ml}
     \int_{M}|e_\lambda|^4 \, dx \le C_\varepsilon \lambda^\varepsilon \|e_\lambda\|_{L^2}^2 \sup_{\gamma \in \Pi} \int_{\gamma} |e_\lambda|^2 \, ds,
\end{equation}
where $\Pi$ denotes the set of unit-length geodesic segments in $M$.
On the other hand, one can formulate a variable-coefficient Mizohata--Takeuchi conjecture for the Carleson--Sj\"olin operator
\[
S_\lambda f(x) := \lambda^{\frac{1}{2}} \int_M e^{-i\lambda d_g(x,y)} a(x,y) f(y) \, dy,
\]
which is the main term in the approximate spectral projector (see \eqref{eq refor}). After rescaling and a routine calculation, Conjecture \ref{VTM} would yield
\begin{equation}\label{eq:osml}
    \int_{M} |S_\lambda f|^2 w \, dx \le C_\varepsilon \lambda^{\varepsilon} \|f\|_{L^2}^2 \sup_{\gamma \in \Pi} \int_{\gamma} w \, ds.
\end{equation}
Recall that $S_\lambda$ essentially reproduces the eigenfunction $e_\lambda$. Thus, taking $f=e_\lambda$ and $w=|e_\lambda|^2$ recovers \eqref{eq:ml}. At first glance, it may seem surprising that \eqref{eq:osml} has a bilinear flavor. However, this is consistent with the more refined bilinear Kakeya--Nikodym estimates established in \cite{MiaoSoggeXiYang2016}. Moreover, choosing $w=|e_\lambda|^4$ in \eqref{eq:osml} gives a direct link between the critical $L^4(\gamma)$ restriction norm of $e_\lambda$ and its critical $L^6(M)$ norm:
\begin{equation}\label{eq:osml'}
    \int_{M} |e_\lambda|^6\, dx \le C_\varepsilon \lambda^{\varepsilon} \|e_\lambda\|_{L^2}^2 \sup_{\gamma \in \Pi} \int_{\gamma} |e_\lambda|^4 \, ds.
\end{equation}
Since $\delta_1(4)=\tfrac{1}{4}$ and $\delta_2(6)=\tfrac{1}{6}$, any improvement in the $L^4(\gamma)$ restriction bound yields a corresponding improvement in the $L^6(M)$ bound. These connections provide further evidence for the naturalness and strength of Conjecture \ref{VTM}.
\end{remark}

\section{Sharpness of main results}\label{sec 6}
\subsection{Sharpness of Theorem \ref{theomain}}
The sharpness of Theorem \ref{theomain} for large $p$ on $S^2$ follows from Theorem 1.5 in \cite{EP}. In this regime, the exponent $\delta(\alpha,p)=\tfrac12-\tfrac{\alpha}{p}$ captures the behavior of zonal spherical harmonics concentrating at a typical point in the support of an $\alpha$-dimensional measure, which shows that Theorem \ref{theomain} is optimal there.

It remains to establish sharpness for the smaller ranges, namely $\alpha\in(\tfrac12,1]$, $p\in[2,4\alpha]$, and $\alpha\in(1,2]$, $p\in[2,2\alpha+2]$. In both cases, the bounds are saturated by highest-weight spherical harmonics.

Recall that the highest-weight spherical harmonic $\mathcal H_k$ is the restriction to $S^2=\{x\in\mathbb R^3:|x|=1\}$ of the harmonic polynomial $k^{1/4}(x_1+i x_2)^k$. A direct computation shows that $\mathcal H_k$ has $L^2$ mass essentially concentrated in a $k^{-1/2}$-tubular neighborhood of the equator $\gamma_0$, where $|\mathcal H_k|\approx k^{1/4}$.

In particular, when $\alpha\in(\tfrac12,1]$ and $p\in[2,4\alpha]$, we have $\delta(\alpha,p)=\tfrac14$. If $\mu$ is any $\alpha$-dimensional probability measure supported on $\gamma_0$, then
\[
\|\mathcal H_k\|_{L^p(M;\,d\mu)}\approx k^{1/4},\qquad p\ge2,
\]
so the bound in Theorem \ref{theomain} is sharp up to a $k^\varepsilon$ loss in this range.

The situation is a bit more involved in the case $\alpha \in (1,2]$, $p \in [2, 2\alpha+2]$, for which $\delta(\alpha, p) = \frac{1}{4} - \frac{\alpha - 1}{2p}$. We shall prove the following.

\begin{proposition}
For each $\alpha\in(1,2)$ there exist an $\alpha$-dimensional measure $\mu$ and a subsequence $k_n\to\infty$ such that
\[
\|\mathcal H_{k_n}\|_{L^p(M;\,d\mu)}\approx k_n^{\frac14-\frac{\alpha-1}{2p}},\qquad p\ge2.
\]
\end{proposition}

\begin{proof}
Fix $k$ and construct a measure $\mu_k$ that captures the behavior of $\mu$ at the resolution $k^{-1}$, since the eigenfunction $\mathcal H_k$ is approximately constant at this scale. The full construction of an $\alpha$-dimensional measure $\mu$ is then obtained by a standard inductive procedure along a sparse sequence of scales.

Work in Fermi coordinates $(s,y)$ about a unit-length segment of the equator $\gamma_0$, identifying the $1/10$-neighborhood of $\gamma_0$ with the rectangle $[0,1]\times[0,1/10]$ in the $s$–$y$ plane. To construct $\mu_k$ at resolution $k^{-1}$, we select $\sim k^\alpha$ squares of side length $k^{-1}$ inside this rectangle to serve as the support of $\mu_k$.

Set
\[
\Omega_{\alpha}^k := \bigcup_{1 \le j \le N} [0,1] \times \big[ k^{-1} j^{\frac{1}{\alpha - 1}},\ k^{-1} \big( j^{\frac{1}{\alpha - 1}} + 1 \big) \big],
\]
where $N = \lfloor \tfrac{1}{10} k^{\alpha - 1} \rfloor$.

\begin{figure}[h]
\centering
\begin{tikzpicture}[scale=10]

    \def\tube{0.2} 
    \def\k{0.025} 
    \def\n{10} 
    
    \draw[blue, thick] (-0.5, 0) -- (0.5, 0);
    \node[blue, below left] at (0.5, 0) {$\gamma_0$};
    
    \draw[red, thick, dashed] (-0.5, \tube) -- (0.5, \tube);

    \fill[red, opacity=0.1] (-0.5, 0) rectangle (0.5, \tube);

    \fill[green!70!black, opacity=0.5] (-0.5, 0.01) rectangle (0.5, 0.01 + \k); 
    \fill[green!70!black, opacity=0.5] (-0.5, 0.07) rectangle (0.5, 0.07 + \k); 
    \fill[green!70!black, opacity=0.5] (-0.5, 0.15) rectangle (0.5, 0.15 + \k); 
    \fill[green!70!black, opacity=0.3] (-0.5, 0.25) rectangle (0.5, 0.25 + \k); 
    \fill[green!70!black, opacity=0.3] (-0.5, 0.37) rectangle (0.5, 0.37 + \k); 
    \draw[<->, black] (-0.55, 0.07) -- node[left] {${k}^{-1}$} (-0.55, 0.07 + \k);

    \draw[<->, red] (0.55, 0) -- node[right] {${k}^{-1/2}$} (0.55, \tube);

\end{tikzpicture}
\caption{The $k^{-1/2}$-tubular neighborhood of $\gamma_0$ and the support of $\mu_k$.}\label{fig: muk}
\end{figure}

It is clear that $\Omega_\alpha^k$ comprises $N$ horizontal rectangles, each of dimensions $1\times k^{-1}$. Each rectangle contains $k$ disjoint $k^{-1}\times k^{-1}$ squares in the $s$–direction. Therefore, the total number of $k^{-1}\times k^{-1}$ squares in $\Omega_\alpha^k$ is
\[
\text{Number of squares} = N \times k \approx \tfrac{1}{10}k^{\alpha-1}\times k = \tfrac{1}{10}k^{\alpha}.
\]

To reflect $\alpha$–dimensional scaling at resolution $k^{-1}$, assign to each such square mass $k^{-\alpha}$ under $\mu_k$. Hence
\[
\mu_k(\Omega_\alpha^k) = (\text{Number of squares}) \times (\text{mass per square})
\approx \tfrac{1}{10}k^{\alpha}\times k^{-\alpha} = \tfrac{1}{10}.
\]

We next compute $\|\mathcal H_k\|_{L^p(M;\,d\mu_k)}$. Recall that $\mathcal H_k$ concentrates in a $k^{-1/2}$–tubular neighborhood of $\gamma_0$, where $|\mathcal H_k|\approx k^{1/4}$. First, determine how many rectangles in $\Omega_\alpha^k$ lie within this $k^{-1/2}$–neighborhood (see Figure \ref{fig: muk}). For the rectangle indexed by $j$, its vertical position is approximately
\[
y_j = k^{-1} j^{\frac{1}{\alpha-1}}.
\]
Those within the $k^{-1/2}$–neighborhood satisfy $y_j\le k^{-1/2}$, i.e.
\[
k^{-1} j^{\frac{1}{\alpha-1}} \le k^{-1/2}
\quad\Longleftrightarrow\quad
j^{\frac{1}{\alpha-1}} \le k^{1/2}
\quad\Longleftrightarrow\quad
j \le k^{\frac{\alpha-1}{2}}.
\]
Thus, the number of such rectangles is $\approx k^{\frac{\alpha-1}{2}}$.

Each of these rectangles contains $k$ squares of size $k^{-1}\times k^{-1}$ along the $s$–direction. Therefore, the total number of squares inside the $k^{-1/2}$–tube is
\[
N_{\text{tube}} = k^{\frac{\alpha-1}{2}}\times k = k^{\frac{\alpha+1}{2}}.
\]

On each such square, $|\mathcal H_k|\approx k^{1/4}$, and the $\mu_k$–mass of a square is $k^{-\alpha}$. Hence
\[
\int |\mathcal H_k|^p\,d\mu_k
\approx N_{\text{tube}} \times |\mathcal H_k|^p \times k^{-\alpha}
= k^{\frac{\alpha+1}{2}} \times k^{\frac{p}{4}} \times k^{-\alpha}
= k^{\frac{p}{4}-\frac{\alpha-1}{2}}.
\]
Taking the $p$–th root yields
\[
\|\mathcal H_k\|_{L^p(M;\,d\mu_k)} \approx k^{\frac{1}{4}-\frac{\alpha-1}{2p}}.
\]

Finally, to construct the desired $\alpha$–dimensional measure $\mu$ on $M$, take a sparse sequence of scales $k_n=2^{2^n}$ and define $\mu_{k_n}$ as above at each scale. By arranging the supports $\Omega_\alpha^{k_n}$ to form a nested sequence (for example, by taking intersections $\Omega_n = \bigcap_{m=1}^n \Omega_{\alpha}^{k_m}$), the measures $\mu_{k_n}$ converge weakly to a limit measure $\mu$ as $n\to\infty$. The resulting $\mu$ is $\alpha$–dimensional and supported near $\gamma_0$. Consequently,
\[
\|\mathcal H_{k_n}\|_{L^p(M;\,d\mu)} \approx k_n^{\frac{1}{4}-\frac{\alpha-1}{2p}},
\]
for all $p\ge2$, completing the proof.
\end{proof}

\subsection{Sharpness of Theorem \ref{theo:cur}}

We use a construction due to Tacy \cite{Tacy} to prove the sharpness of Theorem \ref{theo:cur}. We first recall:

\begin{proposition}[Proposition 3.2 in \cite{Tacy}]
Let $0\le \beta\le \tfrac12$. For each $\lambda\gg1$ in the spectrum on $S^2$, there exists an $L^2$–normalized eigenfunction $e_\lambda$ and a region of size $\lambda^{-1+2\beta}\times\lambda^{-1+\beta}$ on which
\[
|e_\lambda|\gtrsim \lambda^{\frac{1-\beta}{2}}.
\]
This region is “rectangular’’ in the sense that it is a $\lambda^{-1+\beta}$–neighborhood of a $\lambda^{-1+2\beta}$ segment of a great circle on $S^2$.
\end{proposition}

\begin{figure}[h]
\centering
\begin{tikzpicture}[scale=15]
    \def\de{0.04}
    \def\dx{0.24}
    \def\dy{0.05}
    \draw[blue, thick, domain=-0.2:0.2, samples=100] plot (\x, {\x*\x});
    \node[blue, below left] at (-0.15, 0.0225) {$\gamma$};
    \draw[red, thick, domain=-0.2:0.2, samples=100] plot (\x, {\x*\x + \de});
    \draw[red, thick, domain=-0.2:0.2, samples=100] plot (\x, {\x*\x - \de});
    \fill[red, opacity=0.1] (-0.2, {-0.2*-0.2 - \de}) -- plot[domain=-0.2:0.2, samples=100] (\x, {\x*\x + \de}) -- (0.2, {0.2*0.2 + \de}) -- plot[domain=0.2:-0.2, samples=100] (\x, {\x*\x - \de}) -- cycle;
    \node[red, right] at (0.2, {0.2*0.2 + \de}) {$T_{\lambda^{-1+s}}(\gamma)$};
    \draw[green!70!black, thick] (-\dx/2, -\dy/2) rectangle (\dx/2, \dy/2);
    \fill[green!70!black, opacity=0.2] (-\dx/2, -\dy/2) rectangle (\dx/2, \dy/2);
    \draw[<->] (\dx/2 + 0.005, -\dy/2) -- node[right] {$\lambda^{-1+\frac{1+s}{4}}$} (\dx/2 + 0.005, \dy/2);
    \draw[<->] (-\dx/2, -\dy/2 - 0.005) -- node[below] {$\lambda^{-1+\frac{1+s}{2}}$} (\dx/2, -\dy/2 - 0.005);
\end{tikzpicture}
\caption{$T_{\lambda^{-1+s}}(\gamma)$ and the largest $\lambda^{-1+2\beta}\times\lambda^{-1+\beta}$ region that fits inside it.}
\label{fig:tubular_neighborhood}
\end{figure}

Now let $\gamma$ have nonvanishing geodesic curvature and consider its $\lambda^{-1+s}$–tubular neighborhood with $\tfrac13\le s\le \tfrac12$. The largest “rectangular’’ region of the form $\lambda^{-1+2\beta}\times\lambda^{-1+\beta}$ that fits inside this tube has dimensions $\lambda^{-1+\frac{1+s}{2}}\times \lambda^{-1+\frac{1+s}{4}}$ (see Figure \ref{fig:tubular_neighborhood}). Hence,
\[
\lambda^{1-s}\int_{T_{\lambda^{-1+s}}(\gamma)} |e_\lambda|^2\,dx
\gtrsim
\lambda^{1-s}\cdot \lambda^{\,1-\frac{1+s}{4}}\cdot \lambda^{-1+\frac{1+s}{2}}\cdot \lambda^{-1+\frac{1+s}{4}}
= \lambda^{\frac{1-s}{2}}.
\]
Therefore,
\[
\|e_\lambda\|_{L^2_{\mathrm{avg}}(T_{\lambda^{-1+s}}(\gamma))}
\gtrsim \lambda^{\frac{1-s}{4}},
\]
which proves the sharpness of Theorem \ref{theo:cur} for $\tfrac13\le s\le \tfrac12$. When $s=\tfrac13$, the lower bound is $\lambda^{1/6}$, matching the restriction estimate \eqref{eq:5} for $\gamma$. Consequently, \eqref{eq:0 to 1/3} is also sharp for $0\le s\le \tfrac13$.

	\bibliographystyle{alpha}
\bibliography{reference}

\end{document}